\documentclass[12pt,oneside,leqno]{amsart}
\usepackage[margin=1in]{geometry}
\usepackage{amsmath}
\usepackage{amsthm}
\usepackage{amsfonts}
\usepackage{amssymb}
\usepackage{stmaryrd}
\SetSymbolFont{stmry}{bold}{U}{stmry}{m}{n}
\usepackage[hidelinks]{hyperref}
\usepackage{tikz}
\usetikzlibrary{cd}
\usepackage{mathrsfs}

\newenvironment{nouppercase}{%
  \renewcommand{\uppercasenonmath}[1]{}}{}

\usepackage{comment}

\mathchardef\ordinarycolon\mathcode`\:
\mathcode`\:=\string"8000
\begingroup \catcode`\:=\active
  \gdef:{\mathrel{\mathop\ordinarycolon}}
\endgroup

\makeatletter
\renewcommand{\eqref}[1]{\textup{(\ignorespaces\ref{#1}\unskip\@@italiccorr)}}
\def\maketag@@@#1{\hbox{\m@th\normalfont\bfseries#1}}
\makeatother

\numberwithin{equation}{section}
\numberwithin{figure}{section}

\theoremstyle{plain}
\newtheorem{thm}[equation]{Theorem}
\newtheorem*{FundClaim*}{Fundamental Claim}
\newtheorem{lemma}[equation]{Lemma}
\newtheorem{cor}[equation]{Corollary}
\newtheorem{prop}[equation]{Proposition}

\theoremstyle{definition}
\newtheorem{definition}[equation]{Definition}
\newtheorem{question}[equation]{Question}

\newtheorem{example}[equation]{Example}
\newtheorem{axiom}[equation]{Axiom}

\newcommand{\Q}{\ensuremath \mathbb{Q}}
\newcommand{\R}{\ensuremath \mathbb{R}}

\newcommand{\Z}{\ensuremath \mathbb{Z}}
\newcommand{\N}{\ensuremath \mathbb{N}}

\DeclareMathOperator{\End}{End}

\DeclareMathOperator{\im}{im}

\DeclareMathOperator{\ann}{ann}

\DeclareMathOperator{\dom}{dom}
\newcommand{\power}{\mathscr P}

\begin{document}

\title[Algebraization of infinite summation]{Algebraization of infinite summation}

\author{Pace P.\ Nielsen}
\address{Department of Mathematics, Brigham Young University, Provo, UT 84602, USA}
\email{pace@math.byu.edu}

\keywords{endomorphism ring, finite topology, gapless limits, infinite summation, sequential coreflection, series}
\subjclass[2020]{Primary 08A65, 16S50, Secondary 16W80, 22A05, 40A05, 40A10}

\begin{abstract}
An algebraic framework in which to study infinite sums is proposed, complementing and augmenting the usual topological tools.  The framework subsumes numerous examples in the literature.  It is developed using many varied examples, with a particular emphasis on infinitizing the usual group and ring axioms.  Comparing these examples reveals that a few key algebraic properties play a crucial role in the behaviors of different forms of infinite summation.  Special attention is given to associativity, which is particularly difficult to properly infinitize.  In that context, there is an important technique called the Eilenberg-Mazur swindle that is studied and greatly generalized.

Some special properties are singled out as potential axioms.  Interactions between these potential axioms are analyzed, and numerous results explore how to impose new axioms while retaining old ones.  In some cases the axioms classify or categorize a given example.  Surprisingly, such axiomatizations are very concise, relying on only a handful of natural conditions.

These investigations reveal more precisely the part that topology plays in the formation of infinite sums.  Special attention is given to the methods of partial summation and unconditional summation.  In the opposite direction, it is proved that from the infinite sums alone one can create a refined topology, lying between the original topology and its sequential coreflection.  Another especially interesting application of these ideas is the construction of new algebraic quotient structures that retain the ability to handle infinite summation.
\end{abstract}

\begin{nouppercase}
\maketitle
\end{nouppercase}

\section{Introduction}\label{Section:Introduction}

Infinite sums occur naturally throughout mathematics.  Infinite series over $\R$ provide a standard example, and this form of infinite summation generalizes algebraic binary addition.  In algebra, normally one can only add finitely many things, and so infinite sums are classically handled by topological or measure-theoretic methods.  In this paper an algebraic framework is developed to help handle such sums.

This new framework covers a variety of different forms of infinite summation.  A wide range of examples are presented to emphasize its generality.  Surprisingly, among the diversity of examples, certain common conditions arise as touchstones.  Indeed, these conditions and variants thereof have appeared explicitly throughout the literature for decades, such as in the context of ``partial summation of mappings'' \cite{Kuros},  ``summation methods'' \cite{KatzStraus}, ``agreeable categories'' \cite{Corner}, ``partially additive monoids'' \cite{ArbibManes1978}, ``$\Sigma$-monoids'' \cite{Higgs} (later expanded to other $\Sigma$-structures; see \cite{AndresMHeunen} for an overview of the recent relevant literature), and quite recently ``summability spaces'' \cite{BKKPS} and ``summation structures'' \cite{AriManny}.  Though we shall call these common conditions ``axioms'' the reader should think of them only as \emph{potential} axioms, since different forms of infinite summation may satisfy different combinations of the conditions, or none of the conditions at all.

A significant portion of the paper is dedicated to motivating these axioms, including exploring their consequences as well as their drawbacks.  One such motivation is the goal to infinitize the usual axioms of groups and rings, as much as possible.  It turns out that associativity is the trickiest axiom to handle.  If taken to the extreme, infinite associativity prevents the existence of additive inverses, a well known consequence of the Eilenberg-Mazur swindle.  The technique inherent in the swindle generalizes greatly, as shown in Section \ref{Section:Assoc}.  Not only in that section, but in all results throughout the paper, we make it clear which axioms are needed for the results to hold.  In other words, we peel away, as much as possible, the detritus of irrelevant assumptions and conditions.  Indeed, one of the main results of this paper is a succinct axiomatization of algebraic properties of infinite summation found in endomorphism rings (ultimately laid out in Theorem \ref{Thm:MainBigLeftReord}).  Similarly, the absolutely convergent series, after a minor extension, are described as a maximal system satisfying certain constraints (see Proposition \ref{Prop:AbsoluteSeriesCategorical}).

Another important consequence of this framework is that it clarifies the extent to which topology interacts with infinite sums.  This is described in Sections \ref{Section:TopologyRecovered} through \ref{Section:UncondSums}, with special emphasis given to the methods of partial summation and unconditional summation.  Interestingly, when infinite summation arises from limits of partial sums under a topology, and one forgets the topology, then a refinement of the original topology is often recoverable from the sums alone (see the discussion near the end of Section \ref{Section:PerfectingPartial}).

An important technique of algebra is the formation of factor algebras, using equivalence classes.  The paper ends with an investigation into what is needed for infinite sums to pass to such quotients.

\section{Examples and definitions}\label{Section:DefinitionExamples}

The following example will motivate the main definition of this section.

\begin{example}[Series]
Consider a series $\sum_{n=0}^{\infty}a_n$ over $\R$ that converges in the usual sense.  The summands describe an $\N$-indexed family $\underline{a}=(a_0,a_1,\ldots)\in \R^{\N}$.  We briefly denote the series as $\Sigma\underline{a}=\sum_{n=0}^{\infty}a_n$.  Think of $\Sigma$ as a partial function $\R^{\N}\to \R$, defined only on those families with a convergent corresponding series.  We refer to $\R^{\N}$ as the \emph{intended} domain, but of course the true domain is the set of those families whose series are convergent.

It is standard to allow for series indexed by sets other than $\N$.  It is perfectly acceptable, for example, to consider a series on $\N\setminus\{0\}$, say $\sum_{i=1}^{\infty}b_i$, when such a series converges in the usual sense.  By writing $\underline{b}=(b_1,b_2,\ldots)\in \R^{\N\setminus\{0\}}$, we will similarly denote this series as $\Sigma\underline{b}$.  Indeed, any subset of $\N$, including any finite subset, can act as an indexing set for a possible convergent series.  Thus, we have extended the intended domain of $\Sigma$ to the much larger set
\[
\bigcup_{I\in \power(\N)}\R^I.
\]

Since finite summation in $\R$ is independent of any sort of ordering consideration, the intended domain of $\Sigma$ naturally extends to allow for any finite indexing set (not necessarily a subset of $\N$).  However, such an extension raises a thorny set-theoretical issue: the domain of $\Sigma$ is now a \emph{proper class} rather than a set.  We will treat $\Sigma$ as a class function, but those who want to avoid any mention of classes may hereafter assume that indexing sets are implicitly assumed to belong to some fixed set (such as the power set $\power(\N)$, or, if more adventurous, a Grothendieck universe).

In full generality, we may as well take the intended domain of $\Sigma$ to equal $\bigcup_{I}\R^{I}$, where the union ranges over all sets $I$ (inside some universe, if you prefer).  We then merely need to describe how $\Sigma$ acts with regard to an arbitrary indexing set.  For instance, when presented with the unusual indexing set $\Z$, then we can either declare that $\Sigma$ is simply never defined on $\R^{\Z}$, or clarify how $\sum_{n\in \Z}c_n$ is a meaningful quantity (for instance, by treating it as the sum of the two series $\sum_{n=0}^{\infty}c_n$ and $\sum_{n=1}^{\infty}c_{-n}$ when they are convergent).
\end{example}

With that example in mind, we now make the following very general definition.

\begin{definition}\label{Definition:SummationSet}
Given a set $X$, then a partial function
\[
\textstyle{\Sigma \colon \bigcup_{I}X^I\to X},
\]
where the union runs over all sets $I$, is called a \emph{summation system} on $X$.

A family $\underline{a}=(a_i)_{i\in I}\in X^I$ is \emph{summable} when $\underline{a}\in \dom(\Sigma)$.  In that case its \emph{sum} is $\Sigma\underline{a}$.
\end{definition}

If we restrict a summation system $\Sigma$ to $X^I$, then we have a partial $I$-ary operation $\Sigma|_{X^I}\colon X^I\to X$.  One can think of $\Sigma$ as the conglomeration of such operations, as the index set $I$ varies.  This is the source of the terminology ``summation \emph{system}.''

To determine the sum of some family $\underline{a}\in X^I$, some forms of infinite summation rely on more than the underlying set $X$ and the indexing set $I$.  For instance, even in the case of series when $X=\R$ and $I=\N$, one is implicitly using the fact that $\N$ has an ordering $<$ and that $\R$ has a standard topology $\tau$.  If one wants to make the dependence of $\Sigma$ on any extra structural components more explicit (and functorial) that is easily accomplished.  For instance, for series one could instead take $X=\R\times \{\tau\}$ and take $I=\N\times \{<\}$.  However, in our investigation we will find that such fixes are unneeded.  One benefit of summation systems is that they are notationally minimal and avoid unnecessary complications.

The following two examples demonstrate how other types of infinite summation also fit in this framework.

\begin{example}[Integrals]\label{Example:Integrals}
Putting $X=\R$, consider a family $\underline{a}\in X^I$, where the index set $I$ is required to be a subset of $\R$.  Treating $\underline{a}$ as a function $I\to \R$, put $\Sigma\underline{a}=\int_{x\in I} \underline{a}(x)\, dx$, but only when $\underline{a}$ is an integrable function (in the sense of some fixed integration theory, such as Riemann integrable or Lebesgue integrable).  This summation system allows for summing over mainly uncountably infinite families.
\end{example}

\begin{example}[Endomorphisms]\label{Example:Endos}
Let $R$ be a ring, and let $M_R$ be a right $R$-module.  Fix $X:=\End(M_R)$, and write endomorphisms on the left of elements of $M$.

Given a family of endomorphisms $\underline{a}=(a_i)_{i\in I}\in X^I$, it is summable when, for each $m\in M$, only finitely many members of $\underline{a}$ are nonzero on $m$.  In that case, define the sum $\Sigma\underline{a}$ to be the endomorphism where
\[
(\Sigma\underline{a})(m)=\sum_{i\in I}a_i(m).
\]
In the sum on the right side we may drop all zero summands to make it a finite addition, and thus an element of $M$.
\end{example}

These examples demonstrate that summation systems can have wildly varying behaviors.  Many of these differences can be captured via simple statements about the index sets.  We end this section by posing a few axioms that encapsulate important properties involving what we can say about index sets.  (Remember that these are only \emph{potential} axioms, and are not assumed to hold for all summation systems.)  Perhaps the most natural is the following.

\begin{axiom}[\bf Reindexing invariance\rm ]\label{Axiom:ReindexInvariance}
If $\varphi\colon I'\to I$ is a bijection, and if $(a_i)_{i\in I}$ is a summable family, then $(a_{\varphi(i')})_{i'\in I'}$ is also summable with the same sum.
\end{axiom}

The summation system corresponding to \emph{absolutely} convergent series satisfies this axiom, once we have extended the summation system to allow for any countably infinite indexing set.  On the other hand, by Riemann's ``series theorem'' we know that any conditionally convergent series can be permuted to a divergent series, and it can also be permuted to a conditionally convergent series with any other sum, so Axiom \ref{Axiom:ReindexInvariance} fails badly in that case.

Speaking of permutations, if one wishes to focus only on series indexed by $\N$, then it is natural to consider a weaker version of Axiom \ref{Axiom:ReindexInvariance} limited to permutations rather than bijections, or in other words where we require $I=I'$.  This weakened axiom still separates the absolutely convergent series from the rest.

Integrals are highly sensitive to the spacing of values, and so the summation system from Example \ref{Example:Integrals} does not satisfy Axiom \ref{Axiom:ReindexInvariance}, even when limited to permutations.  On the other hand, the summation that occurs in endomorphism rings does have reindexing invariance.

Reindexing invariance is so central that it may transcend being an axiom, and instead be built directly into the definitional framework being studied, such as in \cite[Definition 2.1]{AndresMHeunen}.

Another key attribute of some forms of summation is captured by the following axiom.

\begin{axiom}[\bf Subfamilies remain summable\rm ]\label{Axiom:SubsSummable}
If $I'\subseteq I$ are sets, and if $(a_i)_{i\in I}$ is a summable family, then $(a_{i})_{i\in I'}$ is also summable.
\end{axiom}

A standard exercise given to students is to show that a series is absolutely convergent if and only if all subseries are convergent (and subsequently absolutely convergent).  Thus, just as we did with Axiom \ref{Axiom:ReindexInvariance}, we can use Axiom \ref{Axiom:SubsSummable} to carve out the class of absolutely convergent series from among the class of all convergent series.  On the other hand, if we restrict our summation system to only allow absolutely convergent series indexed by countably infinite---but not finite---sets, then by fiat Axiom \ref{Axiom:SubsSummable} fails, but Axiom \ref{Axiom:ReindexInvariance} still holds.  A modified version of Axiom \ref{Axiom:SubsSummable}, where we add the hypothesis $|I'|=|I|$, remains true.

Integrable functions generally do not remain integrable when restricted to an arbitrary subset of the domain of definition.  However, Riemann integrable functions remain integrable when restricted to closed subintervals.  Similarly, Lebesgue integrable functions remain integrable when restricted to measurable subsets.  Thus, in some contexts Axiom \ref{Axiom:SubsSummable} can be viewed as a valuable goal, by which we can measure improvement from one summation system to another.

Axiom \ref{Axiom:SubsSummable} holds in endomorphism rings; if only finitely many members of a family of endomorphisms of a module $M$ are nonzero on some $m\in M$, then the same holds true for any subfamily.  It is important to note that here (as well as in most of the other examples) we are implicitly using the fact that, by convention, an empty sum is defined to equal to $0$.

We are now equipped with examples of summation systems where Axioms \ref{Axiom:ReindexInvariance} and \ref{Axiom:SubsSummable} both hold, where they both fail, and where the first holds but not the second.  To finish showing the complete independence of these axioms, we will now give a natural example where the second axiom holds but the first fails.

\begin{example}[Noncommutative magmas]\label{Example:NoncommMagma}
Let $X$ be a set with a binary operation; such algebras are sometimes called \emph{magmas}.  Write the operation as multiplication, for convenience only.  Given any family in $X^{\{0,1\}}$, think of that family as an ordered pair.  Define a summation system $\Sigma$ on such ordered pairs by taking $(x,y)\mapsto xy$.  In other words, our ``summation'' is precisely multiplication.

Extend the definition of $\Sigma$ to allow for singleton families and the empty family, say by having $\Sigma$ send any singleton family to its unique member, and taking $()\mapsto z$ for some distinguished element $z\in X$, when $X\neq \emptyset$.  Leaving $\Sigma$ undefined with respect to all other indexing sets, then subfamilies of summable families are summable (i.e., Axiom \ref{Axiom:SubsSummable} holds).

Reindexing invariance fails if $X\neq \emptyset$, because the only set with cardinality $2$ on which $\Sigma$ is defined is $\{0,1\}$.  Less trivially, reindexing even fails for permutations as long as the multiplication is noncommutative.
\end{example}

\section{Summing small collections}

On account of the fact that infinite summation often arises as an extension of finite addition, it is natural to want to impose group-like axioms on summation systems.  In that vein, it seems simplest to start with associativity.  However, already when generalizing associativity to arbitrary \emph{finite} sums, then as convincingly argued (multiplicatively) in \cite{Poonen}, one should \emph{a priori} assume that the empty sum exists and that it is an additive identity.  Thus, the best starting point seems to be the mild assertion:

\begin{axiom}[\bf Empty sum existence\rm ]\label{Axiom:EmptyExists}
The empty family $()$ belongs to the domain of $\Sigma$.
\end{axiom}

This axiom is entirely benign, in the sense that any summation system on a nonempty set $X$ that does not already satisfy this axiom can be forced to do so, by simply choosing a value for $\Sigma()$ from $X$.  In other words, this axiom merely asserts the existence of a distinguished element in $X$.  Note that empty sum existence is an immediate consequence of assuming that subfamilies remain summable, Axiom \ref{Axiom:SubsSummable}, as long as $\Sigma$ is not the empty function.

If we continue to follow the philosophy inherent in \cite{Poonen}, then we should treat the sum of the empty family as an additive identity, no matter its placement in a sum.  This motivates the following nomenclature.

\begin{definition}
Given any family $\underline{a}=(a_i)_{i\in I}\in X^I$, then the \emph{core} of that family is the subfamily $(a_i)_{i\in I'}$ where
\[
I':=\{i\in I\, :\, a_i\neq \Sigma()\}.
\]
(When Axiom \ref{Axiom:EmptyExists} fails, interpret that inequality as vacuously true.)

If $\underline{a}$ is a subfamily of some other family $\underline{b}$ that has the same core, we say that $\underline{b}$ is a \emph{core-extension} of $\underline{a}$, and that $\underline{a}$ is a \emph{core-restriction} of $\underline{b}$.
\end{definition}

With this terminology in place, we can now easily assert what it means for $\Sigma()$ to act as a summation identity.

\begin{axiom}[\bf Zero means nothing\rm ]\label{Axiom:EmptyAddsNada}
If $\underline{b}$ is a core-extension of $\underline{a}$, then
\begin{equation}\label{Eq:Axiom5Implication}
\underline{b}\in \dom(\Sigma)\ \Longleftrightarrow\ \underline{a}\in \dom(\Sigma),
\end{equation}
and $\Sigma\underline{a}=\Sigma\underline{b}$ when both families are summable.
\end{axiom}

If Axiom \ref{Axiom:EmptyExists} fails, then Axiom \ref{Axiom:EmptyAddsNada} is basically vacuously true, since there are no nontrivial core-extensions.  So, it is only relevant to consider the latter axiom in the presence of the former.

Series summation (with or without absolute convergence, but allowing families indexed by subsets of $\N$) satisfies this axiom when the biconditional in \eqref{Eq:Axiom5Implication} is replaced by just the forward implication.  Integral summation, say in the sense of Lebesgue, fails even that forward implication, simply because zero functions on nonmeasurable sets are not integrable.  However, we can enlarge the collection of integrable functions to satisfy Axiom \ref{Axiom:EmptyAddsNada}, by declaring that a function $f$ is integrable if $f$ has the same core as some Lebesgue integrable function $g$; in which case we take the integral of $f$ equal to the Lebesgue integral of $g$.  Finally, Axiom \ref{Axiom:EmptyAddsNada} holds fully for endomorphism ring summation; we can remove or adjoin zero endomorphisms without changing a sum.

One should expect to see some form of Axiom \ref{Axiom:EmptyAddsNada} take hold whenever $X$ is a magma with a \emph{neutral} or \emph{identity} element, as long as the summation system incorporates that algebraic structure in some way.  One of the weakest forms of Axiom \ref{Axiom:EmptyAddsNada} is simply the assertion that any summable family can be replaced by its core (e.g., see the axiom on neutral elements in \cite[Definition 3.1]{AndresMHeunen}).  However, even when the full axiom fails, it often holds after minor modifications (as in the case of integration).  This is made precise, as follows.

\begin{prop}\label{Prop:ZeroExtens}
Let $X$ be a set with a summation system $\Sigma$ satisfying \textup{Axiom \ref{Axiom:EmptyExists}}.  Assume that any two summable families with the same core have the same sum \textup{(}a weakening of \textup{Axiom \ref{Axiom:EmptyAddsNada})}.  Put
\[
\Sigma':=\left\{(\underline{x},s)\in \textstyle{\bigcup_{I}}X^I\times X \, :\, \text{some core-extension of $\underline{x}$ is $\Sigma$-summable with sum $s$}\right\}.
\]
Then $\Sigma'$ is a summation system on $X$, it extends $\Sigma$, it satisfies the modified \textup{Axiom \ref{Axiom:EmptyAddsNada}} where the biconditional in \eqref{Eq:Axiom5Implication} is replaced by only the forward implication, and $\Sigma'$ is a subsystem of every other such extension.

Moreover, put
\[
\Sigma'' :=\left\{(\underline{x},s)\in \textstyle{\bigcup_{I}}X^I\times X \, :\, \text{the core of $\underline{x}$ is $\Sigma'$-summable with sum $s$} \right\}.
\]
Then $\Sigma''$ is a summation system on $X$, it extends $\Sigma'$, it satisfies \textup{Axiom \ref{Axiom:EmptyAddsNada}} fully, and it is a subsystem of every other such extension.

If $\Sigma$ satisfies \textup{Axiom \ref{Axiom:ReindexInvariance}}, then so do $\Sigma'$ and $\Sigma''$.  The same is true for \textup{Axiom \ref{Axiom:SubsSummable}}.
\end{prop}
\begin{proof}
Note that the stated assumption, about equality of sums when the cores are equal, is necessary in order to have any hope that some extension of $\Sigma$ satisfies Axiom \ref{Axiom:EmptyAddsNada}.  Below, we also demonstrate its sufficiency.

We first prove that the relation $\Sigma'$ is a function.  Given a family $\underline{x}\in X^I$, and given $s,t\in X$, suppose that $(\underline{x},s),(\underline{x},t)\in \Sigma'$.  From the definition of $\Sigma'$, fix core-extensions $\underline{y},\underline{y}'\in \dom(\Sigma)$ of $\underline{x}$ such that $s=\Sigma\underline{y}$ and $t=\Sigma\underline{y}'$.  Note that $\underline{y}$ and $\underline{y}'$ have the same core (namely, the core of $\underline{x}$), and so the assumption of the proposition requires that they have the same sum.  In other words $s=t$, as desired.  Thus, $\Sigma'$ is a summation system on $X$.

Second, it is clear that $\Sigma\subseteq \Sigma'$, since any $\Sigma$-summable family is a core-extension of itself.

Third, we show that the summation system $\Sigma'$ satisfies Axiom \ref{Axiom:EmptyAddsNada} when \eqref{Eq:Axiom5Implication} is only a forward implication.  Assume $\underline{b}\in \dom(\Sigma')$ is a core-extension of some subfamily $\underline{a}$.  From the definition of $\Sigma'$, there is some core-extension $\underline{c}\in \dom(\Sigma)$ of $\underline{b}$.  In particular, $\underline{c}$ is also a core-extension of $\underline{a}$.  Hence, $\underline{a}\in \dom(\Sigma')$, with
\[
\Sigma'\underline{a}=\Sigma\underline{c}=\Sigma'\underline{b}.
\]

Finally, suppose that $\Sigma^{\ast}$ is a summation system on $X$ extending $\Sigma$ and also satisfying Axiom~\ref{Axiom:EmptyAddsNada} when \eqref{Eq:Axiom5Implication} is only a forward implication; we will show that $\Sigma'\subseteq \Sigma^{\ast}$.  Let $(\underline{x},s)\in \Sigma'$.  Fix some core-extension $\underline{y}\in \dom(\Sigma)$ of $\underline{x}$ with $s=\Sigma\underline{y}$.  Now, $(\underline{y},s)\in \Sigma\subseteq  \Sigma^{\ast}$.  Hence, by a direct application of the weakened Axiom \ref{Axiom:EmptyAddsNada}, we must have $(\underline{x},s)\in \Sigma^{\ast}$.

The claims in the second and third paragraphs of the proposition are proved similarly and left to the reader.  (Moreover, assuming Axiom \ref{Axiom:SubsSummable}, then $\Sigma=\Sigma'$.)
\end{proof}

The second paragraph of the previous proposition rarely finds application, because a large number of extraneous zero summands is both distracting and unnecessary.  Additionally, many types of infinite summation are developed under strong limitations on the sizes of indexing sets, and thus automatically avoid such superfluous summing.  Hence, it is no surprise that series summation is not (usually) extended to allow adjoining arbitrarily many zeros.  Yet, there are situations where it is natural to impose no constraints on the number of summands (zero or otherwise), such as with regard to endomorphism ring summation.  There are also situations where irrelevant zero summands are often forced by fiat, in order to provide more regularity to an indexing set (for instance, if we decide that series summation should only allow the single indexing set $\N$).

Next, having addressed empty summation, we now handle singletons.  The obvious axiom to consider is the following.

\begin{axiom}[\bf Singletons sum simply\rm ]\label{Axiom:Singletons}
If $\underline{a}\in X^I$ where $|I|=1$, then $\underline{a}$ is summable and its sum is the unique member of $\underline{a}$.  Informally, $\Sigma(a)=a$ for each $a\in X$.
\end{axiom}

This potential axiom goes by other names in the literature, such as the ``unary sum axiom'' in \cite{Haghverdi}.  Series summation (with or without absolute convergence, but allowing any finite set as an index set) satisfies Axiom \ref{Axiom:Singletons}.  Endomorphism summation also satisfies this axiom.  Integral summation does not satisfy the axiom, since integrating a function defined at only one point yields zero (and also since we only allow subsets of $\R$ as index sets).

There are obvious ways in which Axiom \ref{Axiom:Singletons} may be weakened.  For instance, if concerned about the arbitrariness of $I$, one could instead replace $I$ by a specific singleton indexing set, such as $\{0\}$.  Another natural weakening would be to remove the \emph{conclusion} that $\underline{a}$ is summable, and instead require as part of the \emph{premise} that $\underline{a}$ is summable.  In other words, the requirement that $\Sigma$ be a total (rather than merely partial) function on singletons can be disassociated from the requirement that the summable singletons sum simply.

Problems with totality in summation systems are pervasive.  Convergence considerations for series, and integrability issues for functions, are two important examples.  On the other hand, for finite families the totality problem often disappears.  Thus, we posit:

\begin{axiom}[\bf{Finite-totality}\rm ]\label{Axiom:Totality}
The partial function $\Sigma$ is total on $X^I$ when $|I|<\aleph_0$.
\end{axiom}

For each set $I$ we can speak of \emph{$I$-totality}, and when this holds for all sets of a given cardinality $\kappa$ we can speak of \emph{$\kappa$-totality}.  (Technically, cardinals could be indexing sets, but context will always make the distinction clear.)  Series summation naturally extends to have finite-totality.  Endomorphism summation also has finite-totality.  However, neither of these systems extends in any natural way to allow for $\aleph_0$-totality.  We will return to this issue in the next section.

Axiom \ref{Axiom:EmptyExists} is precisely $\emptyset$-totality, or equivalently $0$-totality.

The summation system from Example \ref{Example:NoncommMagma} has $\{0,1\}$-totality, but fails $2$-totality.  Modifying the example produces the following consequential construction of a system satisfying full totality, except on the empty family.

\begin{example}[Choice functions]\label{Example:Choice}
For each nonempty index set $I$, fix---once and for all---some distinguished element $i_I\in I$ (using the axiom of global choice if available, otherwise limit available index sets $I$ to some collection with a global choice function).  Once that choice is made, when given any family $\underline{a}=(a_i)_{i\in I}$, define $\Sigma\underline{a}=a_{i_I}$.  In essence, this summation system instantiates (some version of) the set-theoretical axiom of choice.
\end{example}

There are simple situations where full totality is enjoyed.  This may occur even in ordinary algebra, as in the following example.

\begin{example}[Direct sums]\label{Example:DirectSums}
Let $R$ be a ring, and let $X$ be the collection of right $R$-modules.  (Thus, in this example we have implicitly extended Definition \ref{Definition:SummationSet} to allow $X$ to be a proper class. Alternatively, restrict $X$ to be the collection of such modules inside some Grothendieck universe.)  Given a family $(M_i)_{i\in I}\in X^I$, we can define its sum as $\bigoplus_{i\in I}M_i$.

One can modify this example by working instead with isomorphism types of modules, in which case the resulting summation system has reindexing invariance.  Hereafter, we will implicitly work with isomorphism types whenever talking about direct sums.
\end{example}

Totality over all subsets of $\N$ is a common occurrence for topological systems.  The following example is representative.

\begin{example}[Knots]
For our purposes, let us say that a knot is a topological embedding of $[0,1]$ into the unit cube, where the endpoints go to $(0,0,0)$ and $(1,1,1)$ respectively, and the interior $(0,1)$ maps to the interior of the cube.  This definition may differ slightly from other sources (see, for instance, \cite{Lickorish}), but it captures the essential nature of tame knots, while allowing wild ones that behave well at a cut.  Some might call this a one-component tangle.

Let $X$ be the set of knots (say, up to ambient isotopy).  Let $\underline{a}=(a_i)_{i\in \N}\in X^{\N}$.  We can define $\Sigma\underline{a}$ by linearly shrinking and translating each $a_n$ so that its endpoints land at
\[
\left(\frac{2^n-1}{2^n},\frac{2^n-1}{2^n},\frac{2^n-1}{2^n}\right)\ \text{ and }\ \left(\frac{2^{n+1}-1}{2^{n+1}},\frac{2^{n+1}-1}{2^{n+1}},\frac{2^{n+1}-1}{2^{n+1}}\right)
\]
respectively, and adjoining the single point $(1,1,1)$.  In other words, we append each successive knot at the appropriate end points, but shrink each one in a manner so that the sum embeds in the unit cube.  The same process works just as well for sums over index sets $I\subseteq \N$ (and other ordered countable index sets).  This type of infinite summation has found many uses, including in a strong non-classification result for wild knots \cite{Kulikov}, and is sometimes called the \emph{connected sum}.
\end{example}

To end this section, let us discuss what it means for a summation system to generalize or extend some form of (binary) addition.  It is clear that series summation extends the usual addition in $\R$.  Summing knots generalizes the usual binary operation of tying two knots in sequence.  Endomorphism ring summation is an obvious generalization of adding two endomorphisms.  Finite direct sums are a special case of infinite direct sums.  Examples exist even without commutativity, as Example \ref{Example:NoncommMagma} demonstrates.  The oddball is integration, but that is because integration does not arise from generalizing a binary addition.  We formalize these observations as follows:

\begin{definition}
Let $X$ be a set with summation system $\Sigma$.  The binary (possibly partial) operation that takes a summable family $(a_0,a_1)\in X^{\{0,1\}}$ and outputs $\Sigma(a_0,a_1)$ is called the \emph{induced addition} for $\Sigma$.
\end{definition}

Any summation system with $2$-totality, or even just $\{0,1\}$-totality, then has an induced addition that is a total binary operation.  Moreover, if the summation system has reindexing invariance, then the induced addition is commutative.

Series summation is a prototypical example of the fact that if a summation system $\Sigma$ arises as the generalization of some binary operation $+$ on $X$, then $\Sigma$ can be extended (possibly by Proposition \ref{Prop:ZeroExtens}) so that the operation $+$ is now the induced addition.

\section{Associativity in summation systems}\label{Section:Assoc}

It is helpful to think of reindexing invariance, Axiom \ref{Axiom:ReindexInvariance}, as an infinite version of commutativity.  Similarly, Axioms \ref{Axiom:EmptyExists} and \ref{Axiom:EmptyAddsNada} together infinitize the notion of an additive identity.  With these axioms in place, and assuming $\{0,1\}$-totality, then the induced addition is total, commutative, and exhibits an additive identity.  Associativity is still missing.

To keep things general, throughout this section, we will use $+$ to denote a partial binary operation on $X$, which may or may not be the induced addition. When $+$ is total, we often want $+$ to also be associative, thus making $(X,+)$ a \emph{semigroup}.  For infinite sums, associativity is surprisingly tricky to handle.  Finite associativity takes advantage of the fact that for each $n\in \N$ the index set $\N_{<n}:=\{0,1,\ldots, n-1\}$ comes with a standard ordering.  The simplest infinite ordered set is $\N$ under its usual order, and even in this case strange behaviors arise when adopting modest forms of associativity.  The most basic form might be the following.

\begin{axiom}[\bf Prefix associativity\rm ]\label{Axiom:WeakAssoc1}
If $\underline{a}=(a_i)_{i\in \N}$ is summable, then so is the left shift $\underline{a}'=(a_{i+1})_{i\in \N}$ and $\Sigma\underline{a}=a_0+\Sigma\underline{a}'$.
\end{axiom}

Informally, by writing infinite summation as infinitely repeated addition, we can think of prefix associativity as asserting that
\[
a_0+a_1+a_2+\cdots = a_0+(a_1+a_2+\cdots)
\]
as long as the sum on the left side exists, where we have merely peeled off the first summand from the rest of the sum.  The following theorem demonstrates a noticeable tension between Axiom \ref{Axiom:WeakAssoc1}, some group properties on addition, and $\N$-totality.

\begin{thm}\label{Thm:WeakSwindle}
Let $(X,+)$ be a right cancellative semigroup, with a summation system $\Sigma$ satisfying \textup{Axiom \ref{Axiom:WeakAssoc1}} \textup{(}restricted to constant families\textup{)}.  If $(a)_{i\in \N}\in X^{\N}$ is summable, then $a+a=a$.  Moreover, if $(X,+)$ has a left identity $0$, then $a=0$.
\end{thm}
\begin{proof}
Assume that the constant family $\underline{a}:=(a)_{i\in \N}$ is summable.  Applying Axiom \ref{Axiom:WeakAssoc1} gives
\[
a+\Sigma\underline{a}=a+(a+\Sigma\underline{a}).
\]
Associativity yields $a+\Sigma\underline{a}=(a+a)+\Sigma\underline{a}$.  By right cancellativity we conclude that $a=a+a$.

If, additionally, $(X,+)$ has a left identity $0$, then $a+a=a=0+a$.  Right cancellativity then entails that $a=0$.
\end{proof}

Due to the interplay between the diverse hypotheses used in Theorem \ref{Thm:WeakSwindle}, a number of different behaviors in summation systems can arise.  First, if $(X,+)$ is a nontrivial abelian group, and if prefix associativity holds, the theorem guarantees that $\N$-totality fails.  This happens, for instance, with series summation and endomorphism ring summation, where the only $\N$-indexed, constant, summable family is the zero family.

Second, still assuming prefix associativity, but now additionally assuming $\N$-totality, then Theorem \ref{Thm:WeakSwindle} prevents $(X,+)$ from being a group.  Often, the culprit is a lack of cancellativity, as is the case for connected sums of knots and direct sums of modules.

In fact, when infinite commutativity is thrown into the mix, cancellativity fails at a more basic level, as described in the following theorem.

\begin{thm}\label{Thm:CancellativityStronglyFails}
Let $(X,+)$ be a magma, with a summation system $\Sigma$ satisfying \textup{Axiom \ref{Axiom:ReindexInvariance} (}for permutations on $\N$\textup{)}, \textup{Axiom \ref{Axiom:WeakAssoc1}}, and $\N$-totality.  For any countable list of elements $x_0,x_1,\ldots\in X$, there exists an element $y\in X$ such that $x_i+y=y$ for each $i\in \N$.
\end{thm}
\begin{proof}
Let $\underline{a}\in X^{\N}$ be a family where each $x_i$ occurs infinitely often.  Put $y:=\Sigma\underline{a}$.

Fixing $i\in \N$, it suffices to show that $x_i+y=y$.  Fix nonnegative integers $k_0<k_1<\ldots$ such that $a_{k_{j}}=x_i$ for each $j\in \N$.  For notational ease, put $k_{-1}:=-1$.  Define a map $\varphi\colon \N\to\N$ by the rule that
\[
\varphi(\ell)=\begin{cases}
k_j & \text{if $\ell=k_{j-1}+1$ for some integer $j\geq 0$},\\
\ell-1 & \text{otherwise}.
\end{cases}
\]
The integers in each interval $[k_{j-1}+1,k_j]$ are cycled by $\varphi$ to the left, where the left endpoint is sent to the right endpoint.  Thus, $\varphi$ is a permutation of $\N$.  By Axioms \ref{Axiom:ReindexInvariance} and \ref{Axiom:WeakAssoc1},
\[
y=\Sigma \underline{a}=\Sigma(a_{\varphi(j)})_{j\in \N}=x_i+\Sigma\underline{a}=x_i+y.\qedhere
\]
\end{proof}

Despite the fact that Axiom \ref{Axiom:ReindexInvariance} fails for knot summation, note that the connected sum of finitely many \emph{tame} knots is commutative, because there are ambient isotopies moving the knots across each other.  In the proof of Theorem \ref{Thm:CancellativityStronglyFails}, the permutation $\varphi$ is effected through finite cycles, so only a finite amount of commutativity is needed locally.  These local ambient isotopies glue together nicely.  Thus, the proof of Theorem \ref{Thm:CancellativityStronglyFails} applies to a countable list of tame knots (but note that the resulting knot $y$ is generally wild).

A version of Theorem \ref{Thm:CancellativityStronglyFails} holds for uncountable lists, after amplifying Axiom \ref{Axiom:WeakAssoc1} to allow for arbitrary infinite limit ordinals as index sets.  Such an extension is left to the motivated reader and holds, for example, for the system of direct sums (of isomorphism types) of modules.  Not every index set comes equipped with a standard well-ordering like $\N$ does, but in the presence of reindexing invariance there is no problem, since any indexing set can be replaced by an ordinal.

Prefix associativity is one of the weakest forms of infinite associativity available.  Interestingly, its formal statement is not based on the idea that parentheses can be moved, as is the case for standard ternary associativity
\[
(a+b)+c=a+(b+c).
\]
Rather, it manifests the implicit principle that an \emph{unparenthesized} sum (when it exists) can have parentheses inserted (a sort of twist to the viewpoint of \cite{Poonen}), which in the ternary case might look like
\[
+(a,b,c)\ \text{ exists }\Longrightarrow\ \left[+(a,b,c)=a+(b+c)\right.\ \text{ and }\ \left.+(a,b,c)=(a+b)+c\right].
\]
Consequently, parentheses can be moved, but this is only guaranteed when the unparenthesized sum exists in the first place.  Taking this idea to its logical conclusion, we obtain:

\begin{axiom}[\bf Insertive associativity\rm ]\label{Axiom:InsertAssociativity}
If $\underline{a}=(a_i)_{i\in I}\in X^I$ is summable, and if $P$ is a partition of $I$, then
\begin{equation}\label{Eq:InsertAssoc}
\Sigma\underline{a}=\Sigma\big(\Sigma(a_j)_{j\in J}\big)_{J\in P}
\end{equation}
and every sum on the right side exists.
\end{axiom}

As stated, this axiom lets us insert parentheses but not necessarily remove them, since an unparenthesized family might not be summable.  Of course, when totality holds, this issue is irrelevant.

One benefit of stating Axiom \ref{Axiom:InsertAssociativity} this way is that it applies in circumstances when totality is not present, such as when working with absolutely converging series, as well as working with endomorphism ring summation.  Indeed, many settings do not allow for the arbitrary removal or rearrangement of parentheses; just consider any situation where an infinite sum of zeros, like in Grandi's classical series
\[
(-1+1)+(-1+1)+(-1+1)+\cdots
\]
makes sense, but not the unparenthesized version.  On the other hand, for systems where sums are defined only if the summands do not interact in any way (such as when summing means taking the union of functions with disjoint domains), then parentheses can be safely removed.  In such cases, equalities like \eqref{Eq:InsertAssoc} are often taken to mean that if either side is defined, then so is the other, contrary to our convention that the existence of only one side implies the existence of the other; take care about such points when reading the literature.

Given any \emph{nonempty} sets $I'\subseteq I$, then there is some partition $P$ of $I$ with $I'\in P$.  (When $I'\neq I$, just take $P=\{I',I\setminus I'\}$.)  Thus, Axiom \ref{Axiom:InsertAssociativity} requires that nonempty subfamilies are summable, and so Axiom \ref{Axiom:SubsSummable} follows if the empty family is also summable.

The outer sum over the index set $P$ on the right side of \eqref{Eq:InsertAssoc} is awkward, in the sense that $P$ itself might not be a ``natural'' index set, even if it is somewhat canonical.  In particular, prefix associativity (when $+$ is the induced addition) is not a consequence of insertive associativity; its outer sum is the induced addition, which is a sum over the index set $\{0,1\}$ rather than the partition $\{\{0\},\N\setminus \{0\}\}$.  (Moreover, there is a reindexing of the set $\N\setminus\{0\}$ to $\N$.)  Typically, insertive associativity is present only when reindexing invariance also holds, and so the partition $P$ can be replaced by any indexing of the partition.  Further, when Axiom \ref{Axiom:EmptyAddsNada} holds, then we can also freely allow some indexes to correspond to empty subsets of $I$.

Consider what happens when we put the previous two paragraphs together.  If we assume Axiom \ref{Axiom:ReindexInvariance}, Axiom \ref{Axiom:EmptyAddsNada}, and Axiom \ref{Axiom:InsertAssociativity}, then together they imply the following merger of the axioms.

\begin{axiom}[\bf Monoid merger\rm ]\label{Axiom:PreAbelian}
If $\underline{a}\in X^I$ is summable, and if $\psi\colon K\to \power(I)$ is a map such that for each $i\in I$ there is a unique $k\in K$ such that $i\in \psi(k)$, then
\[
\Sigma \underline{a}=\Sigma\big(\Sigma (a_i)_{i\in \psi(k)} \big)_{k\in K}
\]
and all sums on the right side are defined.
\end{axiom}

This single axiom captures nearly every other axiom we have defined to this point, and under some very weak nontriviality conditions it captures fully what it means for a summation system to generalize finite addition.

\begin{thm}\label{Thm:FiniteExtensions}
Let $X$ be a set with a nonempty summation system $\Sigma$ satisfying \textup{Axiom \ref{Axiom:PreAbelian}}.  Putting $Y:=\im(\Sigma)$, and letting $\Sigma'$ be the restriction of $\Sigma$ to families over $Y$, then $\Sigma'$ as a summation system on $Y$ satisfies all of numbered axioms previously stated in the paper, except possibly finite-totality and prefix associativity.

Additionally, if the induced addition $+$ of $\Sigma'$ is total and associative, then there is a least summation system $\Sigma''\supseteq \Sigma'$ on $Y$ satisfying \textup{Axiom \ref{Axiom:PreAbelian}} such that $\dom(\Sigma'')$ is closed under taking finite extensions of its elements.  Finite-totality and prefix associativity hold for $\Sigma''$.
\end{thm}
\begin{proof}
We first show that singletons from $Y$ sum simply.  As $\Sigma$ is a nonempty summation system, we know that $Y\neq \emptyset$.  Let $y\in Y$ be arbitrary, and write $y=\Sigma \underline{a}$ for some summable family $\underline{a}\in X^I$, for some index set $I$.  For any singleton set $K=\{k\}$, let $\psi\colon K\to \power(I)$ be the map where $k\mapsto I$.  By Axiom \ref{Axiom:PreAbelian}, we have
\[
y=\Sigma \underline{a}=\Sigma\left(\Sigma \underline{a} \right)=\Sigma (y)=\Sigma'(y),
\]
where the subscripting by the singleton family $K$ was suppressed.

Now $1$-totality holds for $\Sigma'$.  Replacing $X$ by $Y$, we hereafter will assume that $X=Y$ and that $\Sigma=\Sigma'$.  Note that Axiom \ref{Axiom:PreAbelian} still holds for this redefined $\Sigma$.

To show reindexing invariance, let $\underline{a}=(a_i)_{i\in I}\in X^I$ be a summable family, and let $\varphi\colon I'\to I$ be a bijection.  Take $K:=I'$, and define $\psi\colon K\to \power(I)$ by the rule $i'\mapsto \{\varphi(i')\}$.  By Axiom \ref{Axiom:PreAbelian} and the fact that singletons sum simply,
\[
\Sigma \underline{a}=\Sigma\left(\Sigma(a_i)_{i\in \psi(k)}\right)_{k\in K}=\Sigma(\Sigma(a_i)_{i\in \{\varphi(i')\}})_{i'\in I'}=\Sigma(a_{\varphi(i')})_{i'\in I'}.
\]

Insertive associativity is clear, by applying Axiom \ref{Axiom:PreAbelian} where $K$ is any given partition of $I$ and where $\psi$ is the identity map on that partition.  As mentioned previously, insertive associativity also implies that nonempty subfamilies of summable families remain summable.

Next, we deal with the empty family.  Given any summable family $\underline{a}\in X^I$, let $K\supseteq I$ be an arbitrary superset of the index set.  Defining $\psi\colon K\to \power(I)$ by the rule
\[
\psi(k)=\begin{cases}\{k\} & \text{ if $k\in I$},\\
\emptyset & \text{ otherwise,}\end{cases}
\]
then Axiom \ref{Axiom:PreAbelian} yields
\begin{equation}\label{Eq:Computation1}
\Sigma \underline{a}=\Sigma \left(\Sigma(a_i)_{i\in \psi(k)} \right)_{k\in K}.
\end{equation}
In particular, when $I$ is a proper subset of $K$, at least one of the inner summands is an empty family, so $\Sigma()$ exists.  (Incidentally, now Axiom \ref{Axiom:SubsSummable} fully holds, as does $0$-totality.)  Returning to the general case when $I\subseteq K$, the family in the outer sum of the right side of \eqref{Eq:Computation1} is then an arbitrary core-extension of $\underline{a}$, which is summable with the same sum as $\underline{a}$.  Thus, the backwards direction of \eqref{Eq:Axiom5Implication} is true.  The forward direction also holds since subfamilies are summable.  As already noted, core-extensions have the same sum as the original family, so we obtain Axiom \ref{Axiom:EmptyAddsNada}.

Our next job is to construct $\Sigma''$.  Before that, we need some notation.  Given a family $\underline{a}\in X^I$, an index $\ell\notin I$, and an element $x\in X$, let $\underline{a}\#_{\ell}x$ be the new family $\underline{b}\in X^{I\cup\{\ell\}}$ where $b_{\ell}=x$ and $b_i=a_i$ for each $i\in I$.  This notation makes it easier to talk about singleton extensions of families, and hence, via iteration, all finite extensions.

Assume the induced addition $+$ is total and associative.  Define the set
\[
\Sigma_1:=\Sigma \cup \left\{(\underline{a}\#_{\ell} x ,\Sigma\underline{a}+x)\, :\, \underline{a}\in \dom(\Sigma)\cap X^I,\ \ell\notin I,\ x\in X\right\}.
\]
First, we will show that $\Sigma_1$ is a summation system on $X$, and subsequently it will have the same induced addition as $\Sigma$.

Letting $(\underline{b},y),(\underline{b},y')\in \Sigma_1$, the goal is to show that $y=y'$.  One possibility is that $\underline{b}=\underline{a}\#_{\ell}x$ with $y=\Sigma\underline{a}+x$.  If, similarly, $\underline{b}=\underline{a}'\#_{\ell'}x'$ with $y'=\Sigma\underline{a}'+x'$, then there are two subcases to consider.  First, it might happen that $\ell=\ell'$, hence $\underline{a}=\underline{a}'$ and $x=x'$, which implies that $y=y'$.  The second subcase is when $\ell\neq \ell'$.  If we write $\underline{a}=\underline{a}''\#_{\ell'}x'$, then $\underline{a}'=\underline{a}''\#_{\ell}x$.  By Axiom \ref{Axiom:PreAbelian} and the fact that $+$ is associative and commutative,
\[
y=\Sigma\underline{a}+x=(\Sigma\underline{a}''+x')+x=(\Sigma\underline{a}''+x)+x'=\Sigma\underline{a}'+x'=y'.
\]
Another possibility is that $(\underline{b},y')\in \Sigma$, and so
\[
y=\Sigma\underline{a}+x=\Sigma\underline{b}=y',
\]
where the second equality follows from another easy application of Axiom \ref{Axiom:PreAbelian}.  All other possibilities are handled similarly.

Second, we verify Axiom \ref{Axiom:PreAbelian} for $\Sigma_1$.  This holds for free for families in $\dom(\Sigma)$, so it suffices to show the relevant equality for a family $\underline{b}\in \dom(\Sigma_1)\setminus \dom(\Sigma)$.  Write $\underline{b}=\underline{a}\#_{\ell} x$ for some $\underline{a}\in \dom(\Sigma)\cap X^I$, some $\ell\notin I$, and some $x\in X$.   Let $\psi\colon K\to \power(I\cup\{\ell\})$ be a function such that for each element of $I\cup\{\ell\}$, that element belongs to the image of exactly one element from $K$.  Our goal is to prove the equality
\begin{equation}\label{Eq:ToShowTricky}
\Sigma_1\underline{b} = \Sigma_1\left(\Sigma_1 (b_j)_{j\in \psi(k)}\right)_{k\in K}
\end{equation}

There is some unique $k_0\in K$ with $\ell\in \psi(k_0)$.  Note that $(a_j)_{j\in \psi(k_0)\setminus \{\ell\}}$ is a subfamily of $\underline{a}$, hence $\Sigma$-summable.  Thus,
\[
(b_j)_{j\in \psi(k_0)}=(a_j)_{j\in \psi(k_0\setminus \{\ell\}}\#_{\ell} x,
\]
which is $\Sigma_1$-summable with sum $\Sigma(a_j)_{j\in \psi(k_0)\setminus \{\ell\}}+x$.  For any other input in $\psi$, things work out more simply; given $k\in K\setminus\{k_0\}$, then $(b_j)_{j\in \psi(k)}=(a_j)_{j\in \psi(k)}$, which is $\Sigma_1$-summable since it is already $\Sigma$-summable, with $\Sigma_1$-sum $\Sigma (a_j)_{j\in \psi(k)}$.  Consequently, we have now shown that the inner sums on the right side of \eqref{Eq:ToShowTricky} are all defined.

For use shortly, note that since Axiom \ref{Axiom:PreAbelian} holds for $\Sigma$, we have
\begin{equation}\label{Eq:Needed}
\Sigma\underline{a} = \Sigma\left(\Sigma(a_j)_{j\in \psi(k)\cap I}\right)_{k\in K}.
\end{equation}
Now, continuing on, the family indexed by $K$ on the right side of \eqref{Eq:ToShowTricky} is exactly
\begin{equation}\label{Eq:UsedNeeded}
(\Sigma (a_j)_{j\in \psi(k)})_{k\in K\setminus \{k_0\}}\#_{k_0} (\Sigma(a_j)_{j\in \psi(k_0)\setminus\{\ell\}}+x).
\end{equation}
Note that the family to the left of the $\#$ symbol in \eqref{Eq:UsedNeeded} is $\Sigma$-summable, being a subfamily of the $\Sigma$-summable family on the right side of \eqref{Eq:Needed}.  Thus, the family in \eqref{Eq:UsedNeeded} is $\Sigma_1$-summable, because it is the extension of a $\Sigma$-summable family by a singleton from $X$.  Hence, the right side of \eqref{Eq:ToShowTricky} is fully defined and equals
\begin{equation}\label{Eq:FinalRightSide}
\Sigma\left(\Sigma(a_j)_{j\in \psi(k)}\right)_{k\in K\setminus \{k_0\}}+(\Sigma(a_j)_{j\in \psi(k_0)\setminus \{\ell\}}+x).
\end{equation}

On the other hand, by combining \eqref{Eq:Needed} with another easy application of Axiom \ref{Axiom:PreAbelian}, the left side of \eqref{Eq:ToShowTricky} is
\[
\Sigma_1 \underline{b}=\Sigma\underline{a}+x=\left[\Sigma\left(\Sigma(a_j)_{j\in \psi(k)}\right)_{k\in K\setminus \{k_0\}} + \Sigma(a_j)_{j\in \psi(k_0)\setminus \{\ell\}}\right]+x.
\]
By associativity of $+$, this equals the quantity in \eqref{Eq:FinalRightSide}, as needed.

Now, recursively repeating these ideas, we obtain a chain of summation systems
\[
\Sigma=\Sigma_0\subseteq \Sigma_1\subseteq \ldots,
\]
where $\dom(\Sigma_{n+1})$ contains all singleton extensions of elements in $\dom(\Sigma_{n})$, and $\Sigma_{n}$ satisfies Axiom \ref{Axiom:PreAbelian}, for each $n\in \N$.  Taking $\Sigma'':=\bigcup_{n\in \N}\Sigma_n$, then its domain is closed under taking finite extensions of its elements, and it satisfies Axiom \ref{Axiom:PreAbelian}

To finish the proof it suffices to show that $\Sigma''$ is contained in every summation system $\Sigma^{\ast}\supseteq \Sigma$ satisfying Axiom \ref{Axiom:PreAbelian} where $\dom(\Sigma^{\ast})$ is closed under taking finite extensions of its elements.  By recursion, it suffices to show that $\Sigma_{1}\setminus \Sigma_0\subseteq \Sigma^{\ast}$.  Given $(\underline{b},y)\in \Sigma_1\setminus \Sigma_0$, write $\underline{b}=\underline{a}\#_{\ell} x$ for some family $\underline{a}\in \dom(\Sigma)\cap X^I$, some index $\ell\notin I$, and some $x\in X$.

Since $\underline{b}$ is a finite extension of the $\Sigma$-summable family $\underline{a}$, we have $\underline{b}\in \dom(\Sigma^{\ast})$.  The fact that the induced additions for $\Sigma$ and $\Sigma^{\ast}$ agree yields
\[
\Sigma^{\ast}\underline{b}=\Sigma^{\ast}\underline{a}+\Sigma^{\ast}(x)_{j\in \{\ell\}}=\Sigma\underline{a}+\Sigma(x)_{j\in \{\ell\}}=\Sigma\underline{a}+x=\Sigma_1\underline{b},
\]
where the first equality is an easy application of Axiom \ref{Axiom:PreAbelian} for $\Sigma^{\ast}$.

The last sentence of the theorem is now clear.
\end{proof}

The possibility implicit in Theorem \ref{Thm:FiniteExtensions}, that $\im(\Sigma)\subsetneq X$, is quite real.  Consider a set $X$ with a distinguished element $x_0\in X$.  If we take $\Sigma$ to be the constant summation system that sends \emph{every} family to $x_0$, then surjectivity of $\Sigma$ fails if $|X|\geq 2$.  There are, of course, more complicated versions of such collapsing behavior.  In such cases, it is natural to replace $X$ by $\im(\Sigma)$, as we did in the previous proof, in which case surjectivity now holds for the restricted system.

There are situations where the extended system $\Sigma''$ in Theorem \ref{Thm:FiniteExtensions} cannot be extended further without violating some previously stated axiom.  For instance, let $(X,+)$ be a nontrivial {\bf finite} abelian group, and let $\Sigma$ be the obvious summation system defined only on those families with finitely many nonzero entries (by the rule that the sum is just the addition of its nonzero entries).  Any proper extension of $\dom(\Sigma)$ would contain a family with a nonzero entry occurring infinitely often.  After passing to a countable constant subfamily and reindexing by $\N$ (assuming the appropriate axioms allowing these reductions still hold), then Theorem \ref{Thm:WeakSwindle} tells us that prefix associativity cannot hold.

Theorem \ref{Thm:FiniteExtensions} suggests another natural axiom for our consideration, which has the special consequence of finite-totality as long as the empty sum exists.

\begin{axiom}[\bf Additive extension closure\rm ]\label{Axiom:AdditiveExtClosure}
Given $\ell\notin I$ and $x\in X$, if $\underline{a}\in X^I$ is summable, then $\underline{a}\#_{\ell}x$ is summable, with $\Sigma \underline{a}\#_{\ell}x=\Sigma \underline{a}+x$.
\end{axiom}

Note that if $\Sigma$ is nonempty, surjective, and satisfies both Axiom \ref{Axiom:PreAbelian} and Axiom \ref{Axiom:AdditiveExtClosure}, then the induced addition $+$ makes $X$ a commutative \emph{monoid} (a semigroup with an identity).  Further, by Theorem \ref{Thm:FiniteExtensions}, all previously numbered axioms hold.  (Consequently,  it must be the case that $(X,+)$ is not a group or $\aleph_0$-totality fails, by Theorem \ref{Thm:WeakSwindle}.)  In some sense, this is all that can be said about generalizing binary addition into an infinite form.  However, we will see in the next section that by viewing algebra operations in a slightly different ``functorial'' way, other connections to finite addition arise.

To end this section, we mention an intriguing trick involving associativity, called the \emph{Eilenberg-Mazur swindle}, based on work of Mazur in \cite{Mazur} and a lemma attributed to Eilenberg and Mazur in \cite{Bass}.  Taking the viewpoint that associativity means that we can shift parentheses, then we should have the following relationship between two constant sums:
\begin{equation}\label{Eq:ShiftEquality}
\Sigma(a+b)_{i\in \N} = a + \Sigma(b+a)_{i\in \N}.
\end{equation}
In other words, conflating summation with addition, we should be able to shift parentheses to the right, and get
\[
(a+b)+(a+b)+\cdots = a + (b+a)+(b+a)+\cdots.
\]
With $\N$-totality (to guarantee that the unparenthesized sum exists) together with insertive associativity and a little reindexing invariance, then \eqref{Eq:ShiftEquality} holds.  The swindle is expressed as follows.

\begin{thm}\label{Thm:NoAdditiveInverse}
Let $(X,+)$ be a magma with an additive identity $0$.  Further, let $\Sigma$ be a summation system where \textup{\eqref{Eq:ShiftEquality}} holds when all sums involved are defined.  If the constant family $(0)_{i\in \N}$ is summable and sums to $0$, then the only element of $X$ with an additive inverse is $0$.
\end{thm}
\begin{proof}
Given $a\in X$, assume that $b\in X$ is an additive inverse.  From the hypotheses on $0$,
\[
\Sigma(a+b)_{i\in \N}=\Sigma(0)_{i\in \N}=0
\]
and
\[
a+\Sigma(b+a)_{i\in \N}=a+\Sigma(0)_{i\in \N}=a+0=a.
\]
The equality $a=0$ follows from \eqref{Eq:ShiftEquality}.
\end{proof}

An immediate consequence of Theorem \ref{Thm:NoAdditiveInverse} is that there is no way to unknot a nontrivial knot by tying another knot afterwards.  Similarly, a nontrivial module cannot be made trivial by (direct) summing it with another module.  In both cases even more is true; a sum is zero if and only if all summands are zero.  Minor modifications of our axioms allow one to prove such facts for these and other summation systems (e.g., see \cite[Proposition 3.10]{AndresMHeunen}).

\section{Negation and functoriality}

In previous sections we discussed at length how to infinitize the notions of commutativity, associativity, and the existence of an additive identity.  The natural next step is to consider additive inverses, which we will write via negations.  Throughout this section, let $(X,+,-,0)$ be an abelian group, which we write more simply as $(X,+)$.

For the moment, suppose that $\Sigma$ is a summation system on $X$ whose induced addition is the given group operation $+$.  Further, to simplify the following discussion, assume that both Axiom \ref{Axiom:PreAbelian} and Axiom \ref{Axiom:AdditiveExtClosure} hold.  In particular, Axiom \ref{Axiom:EmptyAddsNada} holds.

Let $\underline{a}\in \dom(\Sigma)$.  Since $\Sigma\underline{a}$ happens to have an additive inverse, let $\underline{b}$ be an extension of $\underline{a}$ by the single entry $-\Sigma\underline{a}\in X$.  As Axiom \ref{Axiom:AdditiveExtClosure} is in effect, $\underline{b}\in \dom(\Sigma)$.  Moreover, the new sum is $\Sigma\underline{b}=0$.

These calculations suggest that the existence of additive inverses is infinitized when one may adjoin to any summable family a single entry that makes the sum zero.  No new axioms are needed to accomplish this feat, as long as $X$ happens to already be a group under the induced addition.

An alternative perspective arises by viewing the group operations on $X$ in a functorial manner.  To that end, let $I\subseteq J$ be index sets.  We may view
\[
X^I\subseteq X^J
\]
without affecting summations on such families, by using zero extensions.

Before pursuing this perspective further, we note that this viewpoint leads to the formation of a \emph{canonical index set}, capturing all the information in $\Sigma$.  To see this, let $\underline{a}$ be an arbitrary $\Sigma$-summable family.  Assuming, by way of contradiction, that $x\in X\setminus\{0\}$ occurs infinitely often in $\underline{a}$, then the constant family $(x)_{i\in \N}$ is also summable (by passing to a subfamily using Axiom \ref{Axiom:SubsSummable}, and then reindexing by Axiom \ref{Axiom:ReindexInvariance}), which contradicts Theorem \ref{Thm:WeakSwindle}.  Thus, each nonzero element of $X$ can appear only finitely many times in a given summable family, and hence the number of nonzero entries in any summable family is at most
\[
\kappa:=\max(|X|,\aleph_0).
\]
After reindexing, and adding or dropping zero terms as needed, any summable family can be viewed as belonging to $X^{\kappa}$.  In particular, those readers worried about set-theoretical assumptions about classes or universes, but working under the assumptions made in this discussion, can safely work with the \emph{fixed} index set $\kappa$.

Assume now only that $\Sigma$ is a summation system on $X$.  Also assume that $(X,+,-,0)$ is an additive group (although this could be weakened).  Functorially extend the group operations to the direct product $X^I$, for each set $I$, by applying them componentwise.  Abusing notation slightly, we will denote this product group as $(X^I,+,-,0)$, thereby using the same symbols for these componentwise operations, leaving it to the reader to correctly parse different uses of these symbols.  Further, given $\underline{a}\in X^I$ and $\underline{b}\in X^J$, put $\underline{a}+\underline{b}\in X^{I\cup J}$, via zero extensions.  Considering how $\Sigma$ interacts with componentwise operations motivates:

\begin{axiom}[\bf Addition functoriality\rm ]\label{Axiom:FunctorialAdd}
If $\underline{a}$ and $\underline{b}$ are summable, then $\underline{a}+\underline{b}$ is summable and
\[
\Sigma[\underline{a}+\underline{b}]=\Sigma\underline{a}+\Sigma\underline{b}.
\]
\end{axiom}

\begin{axiom}[\bf Negation functoriality\rm ]\label{Axiom:FunctorialNeg}
If $\underline{a}$ is summable, then $-\underline{a}$ is summable and
\[
\Sigma[-\underline{a}]=-\Sigma\underline{a}.
\]
\end{axiom}

These functoriality axioms force $\Sigma$ to be a group homomorphism on $\dom(\Sigma)\cap X^I$, if it is nonempty, with the natural operations making $\dom(\Sigma)\cap X^I$ a group.  Addition functoriality is equivalent to another simple condition, as described in the following proposition.

\begin{prop}\label{Prop:AddFunct}
Let $(X,+)$ be an abelian group with a summation system $\Sigma$.  If $\Sigma$ is surjective and satisfies \textup{Axiom \ref{Axiom:PreAbelian}} and \textup{Axiom \ref{Axiom:AdditiveExtClosure}}, then \textup{Axiom \ref{Axiom:FunctorialAdd}} holds if and only if $\underline{a}+\underline{b}$ is summable whenever $\underline{a}$ and $\underline{b}$ are summable families with distinct index sets.
\end{prop}
\begin{proof}
The forward direction is a tautological weakening.  For the backwards direction, let $\underline{a},\underline{b}\in \dom(\Sigma)$.  By Theorem \ref{Thm:FiniteExtensions}, reindexing invariance holds, and so we can let $\underline{b}'\in \dom(\Sigma)$ be a reindexing of $\underline{b}$ such that its index set is disjoint from the index set of $\underline{a}$.  By hypothesis, $\underline{c}:=\underline{a}+\underline{b}'$ is summable.  Also, finite sums agree with iterated addition, by Axiom \ref{Axiom:AdditiveExtClosure}.

Inserting parentheses in two different ways in the sum $\Sigma\underline{c}$, via two uses of the monoid merger axiom, we obtain the equality needed for addition functoriality.
\end{proof}

When the index sets for $\underline{a}$ and $\underline{b}$ are disjoint, and when $+$ is the induced addition, one can think of Axiom \ref{Axiom:FunctorialAdd} as a form of \emph{flattening}.  In essence, it allows the removal of finitely many parentheses, which is a common property of summation systems found in the literature.

Series summation for absolutely convergent series (allowing reindexing, as well as core-extensions and core-restrictions) satisfies both of the functoriality axioms.  Endomorphism summation also satisfies both axioms.  For these two summation systems, the underlying set $X$ is not only an additive abelian group, but also a ring.  Thus, in both cases, $X^I$ is a ring for each index set $I$, as well as an $X$-$X$-bimodule.  Functoriality axioms could likewise be developed for scaling and multiplication, but in Sections \ref{Section:Distributivity} and \ref{Section:EndosReorderable} we establish even stronger axioms.  Before getting to that, we investigate topological considerations.

\section{Recovering topological information}\label{Section:TopologyRecovered}

In previous sections the focus was on generalizing the abelian group axioms to an infinite setting, with reindexing invariance generalizing commutativity.  This excluded conditionally convergent series summation from most considerations.  In this section the focus is on topology inherent in a summation system, and series summation will play a central, motivational role.  Consequently, much more care is given to the order of summands.

Summation systems are notationally minimal, and in this section it is preferable to preserve this simplicity by defining orderings only implicitly.  We accomplish this task by working with ordinals, which are naturally well-ordered.  Letting ${\rm Ord}$ denote the class of ordinals, and letting $i,j\in {\rm Ord}$, then $j<i$ is synonymous with $j\in i$.  Similarly, $j\leq i$ is synonymous with $j\in i+1$.

In this section and the next we will only work with families indexed by subsets $I\subseteq {\rm Ord}$.  Any such set $I$ is well-ordered by the membership relation.  By the \emph{initial segment} of $I$ determined by $i\in {\rm Ord}$, we mean the set
\[
I_{<i}:=\{j\in I\, :\, j<i\} = I\cap i.
\]
The sets $I_{\leq i}$, $I_{>i}$, and $I_{\geq i}$ are defined similarly.  Note that $I$ is (uniquely) order-isomorphic to a unique ordinal $\sigma(I)\in {\rm Ord}$, which is called its \emph{order type}.  A few results in this paper will be proved by transfinite induction over order types.

Throughout this section $(X,+,-,0)$, or more shortly $(X,+)$, is an abelian group, and $\Sigma$ is a summation system on $X$.  It is not necessarily the case that $+$ is $\Sigma$'s induced addition.  Context will allow readers to differentiate between the additive identity $0\in X$ and the ordinal $0\in {\rm Ord}$.  The motivation for introducing an additive structure on $X$ is that, under the usual addition on $\R$, series summation originates from topology via limits of partial sums.  If we forget the topology, but retain the corresponding summation system, this raises the question of how much of the topology we can reconstruct.

Thankfully, defining limits is straightforward.  Intuitively, we can think of the existence of $\lim_{i\to \infty}s_i$ as saying that the finite sums
\[
s_0-0,\qquad (s_0-0)+(s_1-s_0),\qquad (s_0-0)+(s_1-s_0)+(s_2-s_1),\qquad \ldots
\]
are the partial sums of a convergent infinite summation.  This is a viewpoint taken, for example, in \cite[Definition 4]{KatzStraus}.

Generalizing, we make the following definition of limits indexed by ordinals.

\begin{definition}\label{Definition:Limits}
Let $(X,+)$ be an abelian group with a summation system $\Sigma$.  Given an arbitrary family $\underline{s}=(s_i)_{i\in I}\in X^{I}$ where $I\subseteq {\rm Ord}$, recursively define its \emph{$\Sigma$-limit}, denoted $\lim_{i\in I}s_{i}$ or $\lim \underline{s}$, to be the $\Sigma$-sum of the difference family
\[
d(\underline{s}):= \left(s_{i}-\textstyle{\lim_{j\in I_{<i}}s_{j}}\right)_{i\in I}
\]
when this family exists and is summable, otherwise the $\Sigma$-limit is undefined.
\end{definition}

The axioms introduced in previous sections may be specialized to situations that preserve the ordering of summands.  Further, such specializations often directly affect the properties that $\Sigma$-limits enjoy.  To demonstrate this relationship, consider the following order-preserving weakening of reindexing invariance.

\begin{axiom}[\bf Ordinal reindexing invariance\rm ]\label{Axiom:OrdinalReindexInvariance}
If $I,I'\subseteq {\rm Ord}$, if $\varphi\colon I'\to I$ is an order isomorphism, and if $(a_i)_{i\in I}$ is summable, then $(a_{\varphi(i')})_{i'\in I'}$ is summable with the same sum.
\end{axiom}

If Axiom \ref{Axiom:OrdinalReindexInvariance} holds, then rather than working with subsets $I\subseteq {\rm Ord}$ one could more simply work with elements of ${\rm Ord}$, after replacing $I$ with $\sigma(I)$.  Also, when sums have ordinal reindexing invariance, then $\Sigma$-limits do too.  (The easy proof is left to the reader.)  Note that such ordinal reindexing invariance is the norm for topological limits, and so it is natural to assume Axiom \ref{Axiom:OrdinalReindexInvariance} when one demands that $\Sigma$-limits should have behaviors similar to those of topological limits.

Having specialized Axiom \ref{Axiom:ReindexInvariance}, next comes Axiom \ref{Axiom:SubsSummable}.  Before that, we need to delve deeper into understanding Definition  \ref{Definition:Limits}.  Note that the difference family $d(\underline{s})$ exists if and only if all proper initial segments of $\underline{s}$ have $\Sigma$-limits.  Consequently, all proper initial segments of $d(\underline{s})$ are summable.  In the other direction, we now prove that when all initial segments of a family are summable, then its total sum arises as a ($\Sigma$-)limit of partial sums, just as one might hope is true.

\begin{prop}\label{Prop:LimitsPartialSums}
Let $(X,+)$ be an abelian group with a summation system $\Sigma$, and fix a family $\underline{a}\in X^{I}$ with $I\subseteq {\rm Ord}$.  Assuming that all proper initial segments of $\underline{a}$ are summable, then for each $i\in I$ define the $i$th partial sum as
\[
s_i:=\Sigma (a_j)_{j\in I_{<i}}+a_i.
\]
The partial sum family $p(\underline{a}):=(s_i)_{i\in I}$ has a $\Sigma$-limit if and only if $\underline{a}$ is summable, in which case $\lim_{i\in I}s_i=\Sigma \underline{a}$.  Hence, there is a natural bijection between families whose initial segments are summable and families with $\Sigma$-limits.
\end{prop}
\begin{proof}
Working by transfinite induction on the order type, assume that the needed equality holds for order types smaller than $\sigma(I)$, which guarantees that
\[
\textstyle{\lim_{j\in I_{<i}}}s_j=\Sigma (a_j)_{j\in I_{<i}}
\]
for each $i\in I$.  The difference family of $p(\underline{a})$ is then defined and satisfies
\[
d(p(\underline{a}))=(s_i-\textstyle{\lim_{j\in I_{<i}}}s_j)_{i\in I} = \left(\Sigma (a_j)_{j\in I_{<i}}+a_i-\Sigma (a_j)_{j\in I_{<i}}\right)_{i\in I}=(a_i)_{i\in I}=\underline{a}.
\]
Definitionally, $p(\underline{a})$ has a $\Sigma$-limit exactly when its difference family is summable, in which case the $\Sigma$-limit of $p(\underline{a})$ is $\Sigma d(p(\underline{a}))$, which equals $\Sigma \underline{a}$, thus verifying the needed equality.

The domain of the ``difference family map'' $d$ is the collection of all families whose proper initial segments have $\Sigma$-limits.  Similarly, the domain of the ``partial sums map'' $p$ is the collection of all families whose proper initial segments are summable.  The discussion in the paragraph preceding this proposition verifies the inclusion $\im(d)\subseteq \dom(p)$.  The inclusion $\im(p)\subseteq \dom(d)$ and the equality $d\circ p={\rm id}_{\dom(p)}$ were both verified in the first paragraph of this proof.  Finally, note that $d$ is injective.  (Given distinct families in the domain of $d$, then their difference families disagree at the smallest ordinal index where the two families disagree.)  Putting all that information together shows that $d$ and $p$ are inverse bijections on their respective domains.  The final sentence of the proposition then easily follows by restricting those domains (by removing the word ``proper'' in both cases).
\end{proof}

This motivates the following weakening of Axiom \ref{Axiom:SubsSummable}.

\begin{axiom}[\bf Initial summability\rm ]\label{Axiom:InitialSegments}
Initial segments of ordinal-indexed summable families are summable.
\end{axiom}

According to Proposition \ref{Prop:LimitsPartialSums}, an immediate benefit of Axiom \ref{Axiom:InitialSegments} is that every ordinal-indexed $\Sigma$-sum will contribute an ordered pair to the collection of $\Sigma$-limits.  Otherwise, some of the ordinal-indexed families in $\Sigma$ are unused when constructing $\Sigma$-limits.

For any $\Sigma$-limit to exist, the recursive nature of Definition \ref{Definition:Limits} forces the empty family to be summable.  At first it might seem strange to define $\lim ()=\Sigma ()$, for in a topological space there is no empty limit.  However, when $\Sigma()$ is defined, we may treat $X$ as a pointed space, with $\Sigma()$ as its distinguished point.  In that case it is a simple matter of convention to take the empty limit to equal the distinguished point, whenever convenient.

Potentially, $\Sigma$-limits may be topological limits with respect to more than one topology.  Incredibly, there is always a finest topology where this happens, as a consequence of the following ``fundamental theorem of topological limits.''  It is stated in terms of net limits, rather than merely well-ordered limits, for the sake of generality.

\begin{prop}\label{Prop:FinestTopologyWithLimitsDetermined}
Given a set $X$, let $\mathscr{A}$ be a collection of ordered pairs of the form $(n,x)$, where $n$ is a net in $X$ and $x\in X$.  There is a finest topology such that $x$ is a topological limit of the net $n$, for each such pair.
\end{prop}
\begin{proof}
Let $\tau$ be the set of subsets $U\subseteq X$ such that for each $(n,x)\in \mathscr{A}$, if $x\in U$, then $n$ is eventually in $U$.  This condition is required of open sets in any topology where $x$ is a limit of $n$, for each such pair.  Thus, once we verify that $\tau$ is a topology, it will be the finest one with the desired property.

Vacuously, $\emptyset \in \tau$, while $X\in \tau$ holds trivially.

Next, suppose that $U,V\in \tau$.  Let $(n,x)\in \mathscr{A}$ with $x\in U\cap V$.  Then $x\in U$ and $x\in V$.  Hence, $n$ is eventually in $U$ and eventually in $V$.  Thus, $n$ is eventually in $U\cap V$, starting at any upper bound for both a place where $n$ has become stable in $U$ and a similar place for $V$.

Finally, suppose that $S\subseteq \tau$.  Let $(n,x)\in \mathscr{A}$ with $x\in \bigcup_{V\in S}V$.  Fix $U\in S$ such that $x\in U$.  Then $n$ is eventually in $U$.  Consequently, $n$ is eventually in the union (starting at any place where $n$ has become stable in $U$).
\end{proof}

Topologists might recognize Proposition \ref{Prop:FinestTopologyWithLimitsDetermined}, and its proof, as an alternate formulation of the \emph{final topology} construction \cite[p.\ 32]{Bourbaki}.  We will not pursue that connection in this paper.

We can now recover a topology from a summation system, as follows.

\begin{definition}\label{Def:SigmaTopology}
Let $(X,+)$ be an abelian group and let $\Sigma$ be a summation system on $X$.  The \emph{$\Sigma$-topology}, denoted $\tau_{\Sigma}$, is the finest topology where $\Sigma$-limits are topological limits, meaning that if $x\in X$ is the $\Sigma$-limit of some sequence $\underline{s}$, then $x$ is a $\tau_{\Sigma}$-limit of $\underline{s}$.
\end{definition}

This is the best possible topology to construct from $\Sigma$, if all we assume is that $\Sigma$-limits should have arisen from topological limits.  Moreover, $\tau_{\Sigma}$ has a concrete description; from the proof of Proposition \ref{Prop:FinestTopologyWithLimitsDetermined} we know that $U\in \tau_{\Sigma}$ if and only if whenever $\underline{s}$ has a $\Sigma$-limit in $U$ then a tail of $\underline{s}$ is in $U$.

If we happen to know more about how $\Sigma$ was formed, we could possibly recover an even better topology.  This idea will be explored later in Example \ref{Example:InducedFromTrivial} and in Section \ref{Section:UncondSums}.

Any $\Sigma$-limit is always unique when defined, but topological limits are not generally unique.  We must thus be careful on this point, unless the Hausdorff property holds.  Yet, if there is at least one topology $\tau$ such that $\Sigma$-limits are \emph{unique} $\tau$-limits, then the same is true for $\tau_{\Sigma}$, since $\tau\subseteq \tau_{\Sigma}$.

When working with the $\Sigma$-topology one will generally want to assume Axiom \ref{Axiom:InitialSegments}, else some important information in $\Sigma$ is immediately lost when passing to $\tau_{\Sigma}$.  For instance, if the empty set is not summable, then $\tau_{\Sigma}$ is the discrete topology on $X$, regardless of whether or not there are summable families.  In any case, there is an order-reversing relationship between summation systems and their corresponding topologies.

\begin{lemma}\label{Lemma:TopologyReversal}
Let $(X,+)$ be an abelian group with summation systems $\Sigma$ and $\Sigma'$.  If $\Sigma\subseteq \Sigma'$, then $\tau_{\Sigma}\supseteq \tau_{\Sigma'}$.
\end{lemma}
\begin{proof}
Assuming $\Sigma\subseteq \Sigma'$, then each $\Sigma$-limit is also a $\Sigma'$-limit.  Thus, any topology where $\Sigma'$-limits are topological also has topological $\Sigma$-limits.  Finally, by the fineness statement in Definition \ref{Def:SigmaTopology}, we have $\tau_{\Sigma'}\subseteq \tau_{\Sigma}$.
\end{proof}

We now return to the problem of connecting axioms for summation systems to topological properties.  Recall that partial sums are defined recursively in Proposition \ref{Prop:LimitsPartialSums}, using both the additive group operation and the summation system information.  One might hope that they could be computed more simply as the sum up to, and including, the last summand.  This is expressed axiomatically, as follows.

\begin{axiom}[\bf Postfix associativity\rm ]\label{Axiom:PostfixAssociativity}
Given a set $I\subseteq {\rm Ord}$ with a last element $i_{\rm max}$, if $(a_i)_{i\in I}\in X^{I}$ is summable, then $(a_i)_{i\in I\setminus\{i_{\rm max}\}}$ is summable, and when both are summable
\[
\Sigma(a_i)_{i\in I}=\Sigma(a_i)_{i\in I\setminus\{i_{\rm max}\}} + a_{i_{\rm max}}.
\]
\end{axiom}

This axiom is a sort of reversal of prefix associativity, working from the tail end rather than the front.  It significantly simplifies the computation of any $\Sigma$-limit whose order type is a successor ordinal (matching what happens in a $T_1$ topology), as we now show.

\begin{lemma}\label{Lemma:SuccessorLimits}
Let $(X,+)$ be an abelian group with a summation system $\Sigma$ satisfying postfix associativity.  If $\underline{s}\in X^{I}$ has a $\Sigma$-limit, where $I\subseteq {\rm Ord}$ has a last element $i_{\rm max}$, then that $\Sigma$-limit is the final member $s_{i_{\rm max}}$.
\end{lemma}
\begin{proof}
Assume $\underline{s}$ has a $\Sigma$-limit.  The family $\underline{a}:=d(\underline{s})$ then exists and is summable.  Thus,
\[
\lim \underline{s}=\Sigma \underline{a}=\Sigma (a_i)_{i\in I\setminus\{i_{\rm max}\}} + a_{i_{\rm max}}=s_{i_{\rm max}},
\]
where we use postfix associativity for the middle equality, and we use Proposition \ref{Prop:LimitsPartialSums} for the last equality (since $\underline{s}=p(\underline{a})$).
\end{proof}

If $(a_0)\in X^{\{0\}}$ is summable, postfix associativity is somewhat special, as it tells us that $\Sigma(a_0)=\Sigma()+a_0$.  Additionally, if all $\{0\}$-indexed singletons sum simply, then $\Sigma()$ is the additive identity $0\in X$, as one might hope would be true.

Still assuming $\{0\}$-indexed families sum simply, and also assuming $\{0,1\}$-totality, the case of postfix associativity where $I=\{0,1\}$ tells us that the additive structure $(X,+)$ agrees with the induced addition from $\Sigma$.

Another useful form of associativity can hold in the well-ordered setting.  To that end, let $\underline{a}=(a_i)_{i\in I}\in X^{I}$ be a summable family with $I\subseteq {\rm Ord}$.  Letting $P$ be a partition of $I$ into (disjoint, nonempty) subintervals, then insertive associativity would assert the equality
\[
\Sigma \underline{a} = \Sigma \left(\Sigma (a_i)_{i\in J}\right)_{J\in P}.
\]
After some minor reindexing, one may replace the index set $P$ with any transversal $K$ of the intervals, in which case all index sets are collections of ordinals.  Being a transversal consisting of ordinals, the set $K$ keeps the intervals in the appropriate order.  Putting this all together leads to the following axiom (which follows from Axioms \ref{Axiom:ReindexInvariance} and \ref{Axiom:InsertAssociativity}).

\begin{axiom}[\bf Ordinal insertive associativity\rm ]\label{Axiom:OrdinalAssociativity}
If $\underline{a}=(a_i)_{i\in I}\in X^{I}$ is summable, where $I\subseteq {\rm Ord}$, then
\[
\Sigma\underline{a}=\Sigma\big(\Sigma(a_i)_{i\in J_k}\big)_{k\in K}
\]
for any transversal $K$ of a partition $\{J_k\}_{k\in K}$ of $I$ into subintervals, with $k\in J_k$ for each $k\in K$ (and all sums on the right side exist).
\end{axiom}

A special consequence of this axiom is that any subinterval of an ordinal-indexed summable family is summable, not just initial segments.  Also, this powerful associativity axiom is useful when wanting $\Sigma$-limits to satisfy another desirable property---cofinal subsequences have the same limit---as in the next lemma.

\begin{lemma}\label{Lemma:SubsequenceLimits}
Let $(X,+)$ be an abelian group with a summation system $\Sigma$.  If $\Sigma$ satisfies both \textup{Axiom \ref{Axiom:PostfixAssociativity}} and \textup{Axiom \ref{Axiom:OrdinalAssociativity}}, then
\begin{equation}\label{Eq:SubsequenceProperty}
\text{for each sequence with a $\Sigma$-limit, every cofinal subsequence has the same $\Sigma$-limit.}
\end{equation}
\end{lemma}
\begin{proof}
Assume $\underline{s}\in X^I$ has a $\Sigma$-limit.  Thus, the difference family $\underline{a}:=d(\underline{s})$ exists and is summable.  By transfinite induction, assume \eqref{Eq:SubsequenceProperty} holds when restricted to sequences with order type smaller than $\sigma(I)$.  Let $K$ be cofinal in $I$.  The subfamily $(s_k)_{k\in K}$ has difference family $(s_k -\lim_{\ell\in K_{<k}}s_{\ell})_{k\in K}$, which exists by the inductive hypothesis, since $K_{<k}$ is cofinal in a proper initial segment of $I$ for each $k\in K$. Our goal is to show that the new difference family is summable, with the same sum as the original difference family $\underline{a}$.  This is clear when $I=\emptyset$, so hereafter we assume $I\neq \emptyset$.

The cofinal set $K$ induces a partition of $I$ into subintervals; each $k\in K$ indexes the interval
\[
J_k:=\{i\in I_{\leq k}\, :\, i>j\text{ for each $j\in K_{<k}$}\}.
\]
By Axiom \ref{Axiom:OrdinalAssociativity} it suffices to show, for each $k\in K$, that
\begin{equation}\label{Eq:SubseqInduct}
s_k-\lim_{\ell\in K_{<k}}s_{\ell}=\Sigma\left(a_i\right)_{i\in J_k}.
\end{equation}

Fix $k\in K$, and let $k':=\min J_k$.  By considering the trivial partition of $I$ into a single interval, and using both assumed axioms, we have $\Sigma \underline{a}=\Sigma (\Sigma \underline{a})_{j\in \{k\}}=\Sigma() + \Sigma\underline{a}$.  Thus, $\Sigma()=0$.  In particular, \eqref{Eq:SubseqInduct} holds if $K_{<k}=\emptyset$ by Lemma \ref{Lemma:SuccessorLimits}.  Thus, we may assume $K_{<k}\neq \emptyset$.  Since $K_{<k}$ is cofinal in $I_{<k'}$ and $\sigma(I_{<k'})<\sigma(I)$, then by the inductive assumption
\[
\lim_{\ell\in K_{<k}}s_{\ell}=\lim_{i\in I_{<k'}}s_i=\Sigma (a_i)_{i\in I_{<k'}}.
\]
Using Lemma \ref{Lemma:SuccessorLimits} for the first equality below, using Axiom \ref{Axiom:OrdinalAssociativity} for the second and fourth equalities, using Axiom \ref{Axiom:PostfixAssociativity} for the third equality, and suppressing the indexing on a singleton and on a pair, we find
\begin{align*}
s_{k} & =  \Sigma (a_i)_{i\in I_{\leq k}} = \Sigma(\Sigma (a_i)_{i\in I_{<k'}},\Sigma (a_i)_{i\in J_k}) =  \Sigma(\Sigma (a_i)_{i\in I_{<k'}}) +\Sigma (a_i)_{i\in J_k}\\
& = \Sigma (a_i)_{i\in I_{<k'}} +\Sigma (a_i)_{i\in J_k} =\lim_{\ell\in K_{<k}}s_{\ell}+\Sigma\left(a_i\right)_{i\in J_k}.
\end{align*}
This establishes \eqref{Eq:SubseqInduct}.
\end{proof}

There is a natural barrier that prevents, in almost all cases, an identification of topological limits with $\Sigma$-limits.  Indeed, suppose that topological $\tau$-limits are $\Sigma$-limits, and vice versa.  Given $x_0,x_1\in X$, consider the $\omega+1$ indexed sequence
\[
(x_0,x_1,x_0,x_1,\ldots,x_0).
\]
This clearly has a (topological) limit of $x_0$.  Initial segments of families with $\Sigma$-limits have $\Sigma$-limits, so $\underline{s}:=(x_0,x_1,x_0,x_1,\ldots)$ has a limit, say $x\in X$.  The subsequence $(x_0,x_0,\ldots)$ then has both $x$ and $x_0$ as limits.  Similarly, $(x_1,x_1,\ldots)$ has both $x$ and $x_1$ as limits.  But $\Sigma$-limits are necessarily unique, so $x_0=x=x_1$.  Thus, $|X|=1$.

The issue is that for $\Sigma$-limits to be defined, it is required that all initial subsequences also have (unique) $\Sigma$-limits, but that property is incompatible with general topological limits whenever $|X|\geq 2$.  Thus, to hope for a characterization of $\Sigma$-limits, we must restrict attention to the following special families.

\begin{definition}
Let $(X,\tau)$ be a topological space.  For any set $I\subseteq {\rm Ord}$, a family $\underline{s}\in X^I$ is \emph{gapless} if all nonempty initial subsequences of $\underline{s}$ have $\tau$-limits.
\end{definition}

Even when the summation system $\Sigma$ on $X$ arises from a topology where topological limits are unique, and even when restricting consideration to gapless families, we still do not expect all these gapless limits to be $\Sigma$-limits.  This is because the summation system $\Sigma$ might have been artificially restricted.  For example, consider the standard topology on $\R$ that gives rise to series summation.  The $\omega+1$ sequence
\[
\underline{s}=\left(0,1/2,2/3,\ldots,1\right)
\]
is gapless.  It can be viewed as the sequence of partial sums of the family
\[
\underline{a}=\left(0,1/2,1/6,\ldots, 0\right).
\]
However, if we declare, by fiat, that series summation is restricted to families of length no bigger than $\omega$, then the gapless family $\underline{s}$ has a topological limit but no $\Sigma$-limit.

Artificial restrictions can be avoided by defining sums as often as possible.  We describe how to do so in the next section.

We have discussed specializations of Axioms \ref{Axiom:ReindexInvariance}, \ref{Axiom:SubsSummable}, \ref{Axiom:EmptyExists}, \ref{Axiom:Singletons}, and \ref{Axiom:InsertAssociativity} to the well-ordered setting, as well as their effects on $\Sigma$-limits and the $\Sigma$-topology.  The other axioms behave similarly.  For example, if negation functoriality holds, then negation is continuous under $\tau_{\Sigma}$.  The interested reader is invited to investigate how translation functoriality, a weakening of addition functoriality, similarly influences $\tau_{\Sigma}$.  It seems to be a difficult problem to describe minimal algebraic axioms that would force $(X,+)$ to be a topological group under $\tau_{\Sigma}$.

Changing the implication in the statement of postfix associativity to a biconditional yields a form of additive extension closure, thus adjusting Axiom \ref{Axiom:AdditiveExtClosure} to the well-ordered setting.  One effect of this adjusted axiom is that ``gaps'' in $\Sigma$-limits can only occur at limit ordinals, not successors.

Axiom \ref{Axiom:EmptyAddsNada} likewise specializes, and is connected with the $T_1$ property for topologies.

\begin{prop}
Let $(X,+)$ be an abelian group with a summation system $\Sigma$.  If \eqref{Eq:SubsequenceProperty} holds, if ordinal-indexed singletons sum simply, and if zero entries can be adjoined as the tail of an ordinal-indexed summable family without changing the sum, then $\tau_{\Sigma}$ is $T_1$.
\end{prop}
\begin{proof}
We start with an auxiliary computation.  Let $x\in X$, let $\emptyset\neq I\subseteq {\rm Ord}$, and let $\underline{s}=(x)_{i\in I}$ be a constant sequence with a $\Sigma$-limit.  By transfinite induction on the order type of $I$, it is straightforward to show that $d(\underline{s})$ is just the $I$-indexed core-extension of $(x)_{i\in \{i_0\}}$, where $i_{0}$ is the first element of $I$.  Such tail core-extensions do not change sums, so
\[
\lim \underline{s} = \Sigma d(\underline{s}) = \Sigma (x)_{i\in \{i_0\}} = x,
\]
with the singleton hypothesis used for the last equality.

Let $x\in X$ be arbitrary.  We now show that $X\setminus\{x\}\in \tau_{\Sigma}$.  To that end, assume $y\in X\setminus\{x\}$ is the $\Sigma$-limit of some nonempty family $\underline{s}$.  Further suppose, by way of contradiction, that no tail of $\underline{s}$ is in $X\setminus\{x\}$. Hence, $x$ occurs cofinally in $\underline{s}$.  Using \eqref{Eq:SubsequenceProperty}, after passing to a subsequence if necessary, we may as well assume that all members of $\underline{s}$ are $x$.  Thus, from the work in the previous paragraph, $x=\lim\underline{s} =y\in X\setminus\{x\}$, giving the needed contradiction.
\end{proof}

The natural bijection in Proposition \ref{Prop:LimitsPartialSums} behaves quite well with respect to core-extensions, when assuming some form of Axiom \ref{Axiom:EmptyAddsNada}.  For instance, in the previous proof, a tail core-extension of a singleton $\Sigma$-summable family corresponded to a constant family with a $\Sigma$-limit.  The reader is encouraged to explore what happens when (possibly non-tail) ordinal-indexed core-extensions have the same sum.

\section{Perfecting partial summation}\label{Section:PerfectingPartial}

The method of partial summation generalizes to families indexed by ordinals other than the first countable ordinal $\omega$, as follows.

\begin{definition}\label{Definition:InducedSummationFromTopology}
Let $(X,+)$ be an abelian group, and let $\tau$ be a topology on $X$.  By convention, treat the $\tau$-limit over the empty family as the distinguished element $0\in X$.

Given an arbitrary family $\underline{a}=(a_i)_{i\in I}\in X^I$, where $I\subseteq {\rm Ord}$, recursively define its \emph{induced sum}, denoted $\Sigma\underline{a}$, to be the $\tau$-limit of the family of partial sums
\[
p(\underline{a}):=\left(\Sigma(a_j)_{j\in I_{<i}}+a_i\right)_{i\in I}
\]
when the family exists and has a unique $\tau$-limit, otherwise the sum is undefined.  The function $\Sigma$ is called the \emph{induced summation system} on $X$, induced by $\tau$ and $+$.
\end{definition}

The induced summation system has some familiar properties from Section \ref{Section:TopologyRecovered}.

\begin{thm}\label{Thm:WhatPartialSummationGetsUs}
Let $(X,+)$ be an abelian group, and let $\tau$ be a topology on $X$.  Taking $\Sigma$ to be the induced summation system, then
\begin{itemize}
\item[\textup{(1)}] $\Sigma()=0$,
\item[\textup{(2)}] \textup{Axiom \ref{Axiom:OrdinalReindexInvariance}} holds,
\item[\textup{(3)}] \textup{Axiom \ref{Axiom:InitialSegments}} holds,
\item[\textup{(4)}] \textup{Axiom \ref{Axiom:PostfixAssociativity}} holds, and
\item[\textup{(5)}] a family has a $\Sigma$-limit if and only if all initial segments have unique $\tau$-limits, in which case the $\Sigma$-limit is the $\tau$-limit.
\end{itemize}
If $\tau$ is $T_1$, then further
\begin{itemize}
\item[\textup{(6)}] \textup{Axiom \ref{Axiom:PostfixAssociativity}} holds with the implication changed to a biconditional, and consequently
\begin{itemize}
\item[$\bullet$] ordinal-indexed singletons sum simply,
\item[$\bullet$] the induced addition agrees with $+$, and
\item[$\bullet$] finite-totality holds for ordinal-indexed families, with finite sums given by iterated addition, and
\end{itemize}
\item[\textup{(7)}] any number of zero entries can be adjoined as the tail of an ordinal-indexed summable family without changing the sum.
\end{itemize}
Moreover, if gapless families always have unique $\tau$-limits \textup{(}such as when the topology is Hausdorff\textup{)}, then
\begin{itemize}
\item[\textup{(8)}] property \eqref{Eq:SubsequenceProperty} holds.
\end{itemize}
\end{thm}
\begin{proof}
{\bf (1)}:  This is a consequence of the convention to treat the empty limit as $0$.

{\bf (2)}:  Axiom \ref{Axiom:OrdinalReindexInvariance} follows from the fact that topological limits are invariant under order-preserving reindexing.

{\bf (3)}: Axiom \ref{Axiom:InitialSegments} is immediate from the recursive nature of Definition \ref{Definition:InducedSummationFromTopology}.

{\bf (4)}:  Let $I\subseteq {\rm Ord}$ have a last element $i_{\rm max}$, and let $\underline{a}\in X^I$ be summable.  By definition, $\Sigma \underline{a}$ is the unique $\tau$-limit of the partial sum family $p(\underline{a})$.  The last member of $p(\underline{a})$ is
\[
s_{i_{\rm max}}:=\Sigma (a_i)_{i\in I\setminus\{i_{\rm max}\}}+a_{i_{\rm max}}.
\]
This last member is automatically a $\tau$-limit of $p(\underline{a})$, so by uniqueness $s_{i_{\rm max}}=\Sigma \underline{a}$.

{\bf (5)}:  First, assume that $\underline{s}$ has $\Sigma$-limit $x\in X$.  Definition \ref{Definition:Limits} says that $d(\underline{s})$ exists and $\Sigma d(\underline{s})=x$.  By Definition \ref{Definition:InducedSummationFromTopology}, $p(d(\underline{s}))$ exists and has a unique $\tau$-limit of $x$.  Further, the proof of Proposition \ref{Prop:LimitsPartialSums} yields
\[
p(d(\underline{s}))=\underline{s}.
\]
This shows that $\underline{s}$ has a unique $\tau$-limit which equals its $\Sigma$-limit.  Moreover, recall that for each sequence with a $\Sigma$-limit, all initial segments also have $\Sigma$-limits.  Applying the previous work to initial segments of $\underline{s}$ thus shows that all initial segments of $\underline{s}$ have unique $\tau$-limits.

Conversely, assume that all initial segments of $\underline{s}$ have unique $\tau$-limits, with the $\tau$-limit of $\underline{s}$ equal to $x\in X$.  By transfinite induction on order types, we may assume that all proper initial segments of $\underline{s}$ have $\Sigma$-limits.  In particular $d(\underline{s})$ exists. Again appealing to the proof of Proposition \ref{Prop:LimitsPartialSums}, we have
\[
p(d(\underline{s}))=\underline{s}.
\]
Since $p(d(\underline{s}))$ has a unique $\tau$-limit of $x$, then by Definition \ref{Definition:InducedSummationFromTopology} the family $d(\underline{s})$ is summable, with $\Sigma d(\underline{s})=x$.  Then $\underline{s}$ has a $\Sigma$-limit of $x$ by Definition \ref{Definition:Limits}.  This finishes the inductive step, as well as the verification of (5).

For the remainder of the proof assume that $\tau$ is $T_1$.

{\bf (6)}:  Let $\underline{a}\in X^I$ be an arbitrary summable family.  We show that for any $x\in X$ and any ordinal $\alpha$ larger than all elements in $I$, the extended family $\underline{a}\#_{\alpha}x$ is also summable with sum $\Sigma\underline{a}+x$.  Indeed,
\[
p(\underline{a}\#_{\alpha}x)=p(\underline{a})\#_{\alpha}(\Sigma\underline{a} + x),
\]
which has $\Sigma\underline{a}+x$ as a unique $\tau$-limit; uniqueness uses the $T_1$ property.

Starting with $\Sigma()=0$, and applying the previous paragraph repeatedly, to successively larger finite families, the bullet points follow.

{\bf (7)}: Next, let $\underline{a}\in X^I$ be summable, and let $\underline{b}\in X^J$ be an ordinal-indexed core-extension of $\underline{a}$ such that all the elements of $J\setminus I$ occur after those in $I$.  Working by transfinite induction on the order type of $J\setminus I$, we will show that
\begin{itemize}
\item[(A)] the partial sum sequence $p(\underline{b})$ is just the tail-extension of $p(\underline{a})$ where the new entries are all equal to $\Sigma \underline{a}$, and
\item[(B)] the sequence $p(\underline{b})$ has the same unique limit as $p(\underline{a})$.
\end{itemize}
If $J\setminus I=\emptyset$, then both claims are immediate, since $p(\underline{b})=p(\underline{a})$.  If the order type of $J\setminus I$ is a successor ordinal, then (A) follows from (6) and the inductive assumption.  If the order type of $J\setminus I$ is an infinite limit ordinal, then (A) comes from the recursive definition of the partial sum family and the inductive assumption.  In both situations, (B) comes from the $T_1$ assumption, since the $\tau$-limit of any family with a constant tail has that constant as the unique limit.

(More generally, one can show that a core-extension $\underline{b}\in X^J$ of $\underline{a}\in X^I\cap \dom(\Sigma)$ has the same sum as long as the set $S:=\{I_{<j}\, :\, j\in J\setminus I\text{ and $I_{<j}$ has no greatest element}\}$ is finite.)

{\bf (8)}: Given a sequence with a $\tau$-limit, any cofinal subsequence also has that same $\tau$-limit, even if it may have other $\tau$-limits as well.  Thus, subsequences of gapless sequences are gapless.  Assuming all gapless sequences have unique $\tau$-limits, then all initial segments of gapless families have unique $\tau$-limits.  Thus, gapless families and families with $\Sigma$-limits are the same, by part (5).  The claim about cofinal subsequences always having the same (unique) $\Sigma$-limit is now immediate.
\end{proof}

Assuming better topological properties on $\tau$ can lead to corresponding improvements for the induced summation system $\Sigma$.  The interested reader may wish to investigate how to guarantee that $\Sigma$ respects \emph{all} ordinal-indexed core-extensions.  We also raise:

\begin{question}
Is there a natural topological condition on $\tau$ that implies Axiom \textup{\ref{Axiom:OrdinalAssociativity}} holds for the induced summation system?
\end{question}

We next present concrete examples of induced summation systems, which demonstrate some important behaviors.

\begin{example}\label{Example:InducedFromTrivial}
Let $(X,+)$ be an arbitrary abelian group, and let $\tau=\{\emptyset, X\}$, which is the indiscrete (or trivial) topology on $X$.  Every point $x\in X$ is a $\tau$-limit of every sequence.  (In particular, every sequence is gapless.)  Letting $\Sigma$ be the induced summation system, there are two types of behavior that can occur.

{\bf Case 1}: $|X|=1$.  In this case, every sequence is just a constant sequence of zeros.  Every such sequence has a unique $\tau$-limit, namely the only element $0\in X$.  Thus, the domain of $\Sigma$ consists of all ordinal-indexed families (of zeros), and all sums are $0$.

There is only one topology on $X$.  Hence $\tau_{\Sigma}=\tau$, so $\tau_{\Sigma}$ is the discrete (and trivial) topology.

{\bf Case 2}: $|X|\geq 2$.  In this case, the only sequence with a unique $\tau$-limit is the empty sequence.  Thus, the only $\Sigma$-summable family is the empty family, with sum $0$.

Since the only $\Sigma$-limit is the trivial one, that $\Sigma$-limit is a topological limit (by convention) for each topology on $X$.  The $\Sigma$-topology $\tau_{\Sigma}$ is thus the discrete topology.  Let $\Sigma'$ be the summation system induced by $\tau_{\Sigma}$.  The system $\Sigma'$ is much larger than the (nearly empty) system $\Sigma$; using parts (6) and (7) of Theorem \ref{Thm:WhatPartialSummationGetsUs}, each ordinal-indexed family with only finitely many nonzero members is $\Sigma'$-summable, with the sum obtained by adding those nonzero members.  That these are the only $\Sigma'$-summable families is also easy to show.

The motivation for defining the $\Sigma$-topology as in Definition \ref{Def:SigmaTopology} is to find the ``best'' situation where $\Sigma$-limits are topological.  If we know, \emph{a priori}, that $\Sigma$ arises as an induced summation system, we can try to do even better at recapturing the original topology.  Define the \emph{full-$\Sigma$-topology}, denoted $\tau_{\Sigma}^{\ast}$, to be the topology generated by all those topologies such that $\Sigma$ is the induced summation system.

Returning to our current situation, where $\tau$ is the trivial topology, there are two subcases.

{\bf Case 2a}: $|X|=2$.  In this case, there are exactly four topologies on $X$.  Checking each case, we see that $\tau$ is the only topology inducing $\Sigma$.  Thus, $\tau=\tau_{\Sigma}^{\ast}\subsetneq \tau_{\Sigma}$.

{\bf Case 2b}: $|X|\geq 3$.  For each $x\in X$, the topology $\{\emptyset, \{x\},X\}$ has $\Sigma$ as its induced summation system.  Thus, any singleton is open in the topology these all generate.  Therefore, $\tau_{\Sigma}^{\ast}$ is the discrete topology, so $\tau\subsetneq \tau_{\Sigma}^{\ast}=\tau_{\Sigma}$.
\end{example}

\begin{example}\label{Example:Rchain}
Let $X=\R$, with its usual group structure.  Also let $\tau$ be the standard topology on $\R$, and let $\Sigma$ be the induced summation system.  The $\Sigma$-summable families of order type $\omega$ are exactly the sequences of terms of conditionally convergent series (in the given order), and their sums are given by the corresponding values of the series.  However, there are $\Sigma$-summable families of other order types, both longer and shorter than $\omega$.

By condition (5) of Theorem \ref{Thm:WhatPartialSummationGetsUs}, we know that $\Sigma$-limits are $\tau$-limits.  Thus, $\tau\subseteq \tau_{\Sigma}$.  We finish this example by proving the reverse inclusion.

Assume, $U\notin \tau$.  There then must exist some $x\in U$ such that no standard open interval around $x$ is contained in $U$.  Fix a sequence of points $\underline{s}=(s_i)_{i\in \N}$, with none of its members belonging to $U$, such that $|s_i-x|<(i+1)^{-2}$.  The $\tau$-limit of $\underline{s}$ is $x$.  This is also the $\Sigma$-limit of $\underline{s}$, by Theorem \ref{Thm:WhatPartialSummationGetsUs}(5), and so it is also the $\tau_{\Sigma}$-limit.  Hence $U\notin \tau_{\Sigma}$.   (Moreover, the terms $s_i$ approach $x$ quickly enough to guarantee that $d(\underline{s})$ is absolutely converging.)
\end{example}

In the previous examples, the process of passing back and forth, from a topology $\tau$ to its induced summation system $\Sigma$, and from a summation system $\Sigma$ to the $\Sigma$-topology $\tau_{\Sigma}$, ultimately stabilized.  This motivates defining a special map.

\begin{definition}
Let $(X,+)$ be an abelian group, and let $\mathscr{T}$ be the set of topologies on $X$.  The \emph{induced self-map} is the function $\varphi\colon \mathscr{T}\to \mathscr{T}$ given by the rule $\tau\mapsto \tau_{\Sigma}$, where $\Sigma$ is the induced summation system on $X$ (induced by $\tau$ and $+$).
\end{definition}

Key features of the induced self-map $\varphi$ hold quite generally, as we now prove.

\begin{thm}\label{Thm:BigTopTheorem}
Let $(X,+)$ be an abelian group.  If $\varphi$ is the induced self-map, then:
\begin{itemize}
\item[\textup{(1)}] $\varphi$ does not depend on the choice of abelian group structure on $X$,
\item[\textup{(2)}] if $\tau$ is a topology on $X$, and if $\Sigma$ and $\Sigma'$ are the induced summation systems induced by $\tau$ and $\varphi(\tau)$, respectively, then $\Sigma\subseteq \Sigma'$,
\item[\textup{(3)}] every topology in $\im(\varphi)$ is $T_1$,
\item[\textup{(4)}] $\varphi$ is extensive, meaning $\tau\subseteq \varphi(\tau)$ for each topology $\tau$ on $X$, and
\item[\textup{(5)}] $\varphi$ is idempotent, meaning $\varphi^2=\varphi$.
\end{itemize}
\end{thm}
\begin{proof}
Throughout the proof, let $\tau$ be an arbitrary topology on $X$, and let $\Sigma$ be the induced summation system on $X$ (induced by $\tau$ and $+$).

{\bf (1) and (4)}: By definition, $\tau_{\Sigma}$ is the finest topology where $\Sigma$-limits are topological limits.  However, by Theorem \ref{Thm:WhatPartialSummationGetsUs}(5), the $\Sigma$-limits are exactly the same thing as $\tau$-limits on families whose initial segments have unique $\tau$-limits.  (One can restrict to nonempty limits, if one wants to forget that $X$ is a pointed space.)  Thus, $\tau\subseteq \tau_{\Sigma}$.  Further, this lets us replace all mention of $\Sigma$-limits by information determined solely by the topology $\tau$, without reference to the additive structure.

{\bf (2) and (5)}: Let $\Sigma'$ be the summation system induced by $\tau_{\Sigma}$.  As mentioned above, by Theorem \ref{Thm:WhatPartialSummationGetsUs}(5) the $\Sigma$-limit families are exactly the families whose initial segments have unique $\tau$-limits, with the two types of limits agreeing.  Definition \ref{Def:SigmaTopology} tells us that those are also $\tau_{\Sigma}$-limits, and they remain unique since $\tau_{\Sigma}$ is finer than $\tau$ by part (4).  Finally, by an inductive argument using Definition \ref{Definition:InducedSummationFromTopology} applied to both of the two induced summation systems, we get that any $\Sigma$-summable family is $\Sigma'$-summable, with the sums agreeing.

The inclusion $\varphi(\tau)\subseteq \varphi^2(\tau)$ holds by part (4).  The reverse inclusion holds using $\Sigma\subseteq \Sigma'$ in conjunction with Lemma \ref{Lemma:TopologyReversal}.

{\bf (3)}: Let $x\in X$.  Let $\tau_{x}$ consist of all sets of the form $U$ or $U\setminus \{x\}$, for each $U\in \tau$.  It is straightforward to show that $\tau_x$ is a topology refining $\tau$, under which $x$ is a closed point.

If $x$ is closed in $\tau$, then $\tau_{x}=\tau\subseteq \varphi(\tau)$, and so $x$ is closed in $\varphi(\tau)$.  Now suppose that $x$ is not closed in $\tau$, so $\tau\subsetneq \tau_x$.  It suffices to show that $x$ is closed in $\varphi(\tau)$ in this case too.

Let $\underline{s}$ be any sequence with a $\Sigma$-limit.  Thus, all initial segments of $\underline{s}$ have unique $\tau$-limits.  Consequently, since $x$ is not closed under $\tau$, then $x$ cannot be a member of $\underline{s}$; for, any initial segment ending at $x$ would have at least two limits.  With $x$ not appearing in $\underline{s}$, the $\tau$-limit of $\underline{s}$ is also a $\tau_x$-limit; the convergence property is unaffected when considering any of the new open sets $U\setminus \{x\}\in \tau_x\setminus \tau$.  Since $\varphi(\tau)$ is the finest topology where $\Sigma$-limits are topological, we have $\tau_x\subseteq \varphi(\tau)$.  Hence $x$ is closed in $\varphi(\tau)$.
\end{proof}

\begin{cor}\label{Cor:InducedSelfMapSpecial}
The induced self-map $\varphi$ is a closure operator when restricted to topologies whose limits are unique on gapless families, hence on Hausdorff topologies.  Subject to this restriction, the inclusion in \textup{Theorem \ref{Thm:BigTopTheorem}(2)} is an equality.
\end{cor}
\begin{proof}
Let $\tau\subseteq \tau'$ be topologies, and assume any sequence with multiple $\tau$-limits must have a gap.  Then $\tau'$ has the same property; any sequence with multiple $\tau'$-limits has multiple $\tau$-limits, hence has a $\tau$-gap, thus also a $\tau'$-gap at the same spot.  This shows that $\varphi$ remains a self-map when restricted to topologies whose limits are unique on gapless families.  (An easier argument works for Hausdorff topologies.)

Let $\Sigma$ and $\Sigma'$ be the induced summation systems for $\tau$ and $\tau'$, respectively.  Gapless sequences under $\tau'$ are gapless under $\tau$, so by Theorem \ref{Thm:WhatPartialSummationGetsUs}(5) any family with a $\Sigma'$-limit has a $\Sigma$-limit, and those unique limits are the same (since $\tau\subseteq \tau'$).  Thus, $\Sigma'\subseteq \Sigma$.  Hence, by Lemma \ref{Lemma:TopologyReversal}, we have $\varphi(\tau)\subseteq \varphi(\tau')$.  We already proved that $\varphi$ is extensive and idempotent, and having now shown monotonicity, $\varphi$ is a closure operator.

Taking $\tau'=\varphi(\tau)$, we have the reverse inclusion $\Sigma\subseteq \Sigma'$ by Theorem \ref{Thm:BigTopTheorem}(2).  Hence, equality holds in this case.
\end{proof}

According to Theorem \ref{Thm:BigTopTheorem}(1), the induced self-map can be defined without reference to any additive structure on $X$.  This lets us reorient ourselves by positioning $\varphi(\tau)$ relative to other common constructions, without the need for any algebra.

Towards that end, we now recall three standard topologies, each defined in terms of different kinds of limits.   Let $(X,\tau)$ be an arbitrary topological space.  First, there is a refinement $\tau_{\rm seq}$ of $\tau$, where a set $S\subseteq X$ is $\tau_{\rm seq}$-closed exactly when the $\tau$-limit of any $\omega$-sequence from $S$ still lies in $S$ (when such a limit exists).  It is easily constructed using Proposition \ref{Prop:FinestTopologyWithLimitsDetermined}.  In the literature, the topology $\tau_{\rm seq}$ is called the \emph{sequential coreflection} of $\tau$.

Instead of defining closure under $\omega$-sequence limits, we can define closure under arbitrary ordinal-indexed sequence limits, giving a topology $\tau_{\rm chain}$.  This might be called the \emph{chain-net coreflection}, and $(X,\tau_{\rm chain})$ is a \emph{chain-net} space.  Spaces that are chain-net also go by the names \emph{pseudoradial} or \emph{folgenbestimmte R\"{a}ume}; see \cite[Section D-4]{Encyc}.  Even more generally, defining closure in terms of arbitrary nets (not necessarily well-ordered) yields a topology that we denote as $\tau_{\rm net}$.  The inclusions
\[
\tau_{\rm net}\subseteq \tau_{\rm chain}\subseteq \tau_{\rm seq},
\]
are clear.  Examples are known where either subset may be proper.  The full equality $\tau=\tau_{\rm net}$ is another well-known fact.

Defining closure in terms of sequences whose initial segments have unique $\tau$-limits similarly yields a topology, which is exactly $\varphi(\tau)$. This topology is $T_1$, by Theorem \ref{Thm:BigTopTheorem}(3), and clearly $\tau_{\rm chain}\subseteq \varphi(\tau)$.

When $\tau$ is $T_1$, more is true.  Suppose that $\underline{s}=(s_0,s_1,\ldots)$ is any $\omega$-sequence with some (possibly non-unique) $\tau$-limit $x$.  Then, the new sequence
\[
\underline{s}'=(s_0,x,s_1,x,\ldots)
\]
now has $x$ as its unique $\tau$-limit.  Thus, all nonempty initial segments of $\underline{s}'$ have unique $\tau$-limits.  This means that $x$ must be the $\varphi(\tau)$-limit of $\underline{s}'$.  Since $\underline{s}$ is a cofinal subsequence of $\underline{s}'$, then $x$ is also a $\varphi(\tau)$-limit of $\underline{s}$.  Thus,
\[
\tau=\tau_{\rm net}\subseteq \tau_{\rm chain}\subseteq \varphi(\tau)\subseteq \tau_{\rm seq},
\]
when $\tau$ is $T_1$.  For the standard topology on $\R$, equality holds throughout; a subset of $\R$ is sequentially closed if and only if it is closed.  Thus, from the usual summation system for series, the original topology on $\R$ is fully reconstructible, as shown in Example \ref{Example:Rchain}.

Generally, we can only expect to recover a proper refinement of the original topology $\tau$ sitting above $\tau_{\rm chain}$ (and below $\tau_{\rm seq}$ when $\tau$ is $T_1$).  Some topological properties of $\tau$ pass to $\varphi(\tau)$.  Moreover, they are often reflected as summational properties for the induced summation system, as we have observed.  This raises the following two questions.

\begin{question}
Is there a natural characterization of the topologies belonging to $\im(\varphi)$?
\end{question}

\begin{question}
Is there a natural list of axioms that are necessary and sufficient for a summation system to occur as the induced summation system (for a topology in $\im(\varphi))$?
\end{question}

\section{Unconditional summation}\label{Section:UncondSums}

The method of partial summation is not the only way to use extra information connected with $X$ to form a summation system.  For example, both absolutely converging series summation and endomorphism ring summation arise from a method called \emph{unconditional convergence}.  It is defined, quite generally, as follows.

\begin{definition}\label{Def:UncondSums}
Let $(X,+)$ be an abelian group, and let $\mathscr{A}\subseteq \power(X)$.  A family $\underline{a}\in X^I$ is \emph{unconditionally summable} to $x\in X$ (relative to $\mathscr{A}$) if for each $S\in \mathscr{A}$ there exists a finite set $F_{S,\underline{a}}\subseteq I$, such that for any finite set $F'$ satisfying $F_{S,\underline{a}}\subseteq F'\subseteq I$, then $\sum_{i\in F'}a_i - x \in S$.  (The sum over $F'$ is just iterating the abelian group addition on $X$.)

The summation system
\[
\Sigma_{\mathscr{A}}:=\left\{(\underline{a},x)\in \textstyle{\bigcup_I} X^I \times X\, :\, \text{$x$ is the unique unconditional sum of $\underline{a}$}\right\}
\]
is the \emph{unconditional summation system} (induced by $\mathscr{A}$ and $+$).
\end{definition}

Unconditional summability is sometimes defined only when $I$ is a countable index set, but we will not make that restriction.  Definition \ref{Def:UncondSums} is usually applied when $\mathscr{A}$ is the set of open neighborhoods of $0$, under some topology.  In that case, given an arbitrary index set $I$, the set of finite subsets of $I$ forms a net, and unconditional summation can be viewed in terms of topological limits indexed by such nets.  Absolute convergence is equivalent to unconditional convergence for series, under the usual topology on $\R$.  For general Banach spaces, it is known that absolute convergence implies unconditional convergence, but the converse holds if and only if the space is finite dimensional, by \cite[Theorem 1]{DR}.

Definition \ref{Def:UncondSums} is usually only applied when $S\in\mathscr{A} \Longrightarrow -S\in \mathscr{A}$, else the condition $\sum_{i\in F'}a_i-x\in S$ lacks some symmetry.  We leave it to the reader to show that when this implication holds, the unconditional summation system $\Sigma_{\mathscr{A}}$ satisfies negation functoriality.

We also leave it to the reader to show the following standard facts.  The unconditional summation system $\Sigma_{\mathscr{A}}$ always has reindexing invariance.  It also always has additive extension closure, and moreover the implication in Axiom \ref{Axiom:AdditiveExtClosure} improves to a biconditional.  Consequently, prefix (and postfix) associativity hold.  (Another consequence is that if $\Sigma_{\mathscr{A}}$ is nonempty, it is surjective.)  Finally, Axiom \ref{Axiom:EmptyAddsNada} always holds for $\Sigma_{\mathscr{A}}$ when interpreting ``core-extension'' to mean extending by $0$ entries (rather than by $\Sigma()$, which might not exist).

The following two results provide some interesting criteria regarding uniqueness of unconditional sums.

\begin{prop}\label{Prop:UniqueUncondSums}
Let $(X,+)$ be an abelian group, and let $\mathscr{A}\subseteq \power(X)$.  If $\underline{a}\in X^I$ with $|I|<\infty$, then $x\in X$ is an unconditional sum of $\underline{a}$ if and only if $\sum_{i\in I}a_i -x\in \bigcap_{S\in \mathscr{A}}S$.  Therefore, a finite unconditional sum exists and is unique if and only if $|\bigcap_{S\in \mathscr{A}}S|=1$, in which case $\Sigma_{\mathscr{A}}$ satisfies finite totality.  Consequently, the following conditions are equivalent:
\begin{itemize}
\item empty set existence holds with $\Sigma_{\mathscr{A}}()=0$,
\item singletons sum simply under $\Sigma_{\mathscr{A}}$, and
\item $\bigcap_{S\in \mathscr{A}}S=\{0\}$.
\end{itemize}
\end{prop}
\begin{proof}
Let $\underline{a}\in X^I$ with $I$ finite.  In Definition \ref{Def:UncondSums}, there is no harm in replacing $F_{S,\underline{a}}$ by any finite superset of $F_{S,\underline{a}}$ contained in $I$.  Since $I$ is finite, we may as well take $F_{S,\underline{a}}=I$.  Therefore, $x$ is an unconditional sum of $\underline{a}$ if and only if $\sum_{i\in I}a_i-x\in S$ for each $S\in \mathscr{A}$.  Equivalently, $ \sum_{i\in I}a_i -x\in \bigcap_{S\in \mathscr{A}}S$.

An immediate consequence is that some (equivalently, each) finite unconditional sum exists and is unique if and only if $\bigcap_{S\in \mathscr{A}}S$ is a singleton.  By definition, an unconditional sum contributes an ordered pair to $\Sigma_{\mathscr{A}}$ exactly when the sum is unique.

The equivalence of the final trio of conditions now follows by noting that iterated addition $\sum_{i\in I}a_i$ is $0$ when $I=\emptyset$, and when $I=\{i_0\}$ the iterated addition is $a_{i_0}$.
\end{proof}

\begin{prop}\label{Prop:SecondUniqueUncondSums}
Let $(X,+)$ be an abelian group, and let $\mathscr{A}\subseteq \power(X)$.  Assuming
\begin{itemize}
\item $\bigcap_{S\in \mathscr{A}}S=\{0\}$ and
\item for each $S\in \mathscr{A}$ there exist $T,U\in \mathscr{A}$ with $T-U\subseteq S$,
\end{itemize}
then unconditional sums are unique when they exist.  Moreover, $\Sigma_{\mathscr{A}}$ satisfies addition and negation functoriality.
\end{prop}
\begin{proof}
Assume that $\underline{a}\in X^I$ has unconditional sums $x,y\in X$.  Let $S\in \mathscr{A}$, and fix $T,U\in \mathscr{A}$ with $T-U\subseteq S$.  By Definition \ref{Def:UncondSums}, there exists a finite set $F_{T,\underline{a}}\subseteq I$ such that $\sum_{i\in F'}a_i-x\in T$ for each finite set $F'$ satisfying $F_{T,\underline{a}}\subseteq F'\subseteq I$.  Similarly, there exists a finite set $F_{U,\underline{a}}\subseteq I$ such that $\sum_{i\in F'}a_i-y\in U$ for each finite set $F'$ satisfying $F_{U,\underline{a}}\subseteq F'\subseteq I$.  In particular, taking $F'=F_{T,\underline{a}}\cup F_{U,\underline{a}}$, we find
\[
y-x=\left(\sum_{i\in F'}a_i-x\right)-\left(\sum_{i\in F'}a_i-y\right)\in T-U\subseteq S.
\]
Since $S\in \mathscr{A}$ is arbitrary and $\bigcap_{S\in \mathscr{A}}S=\{0\}$, we have $y-x=0$.  In other words, $x=y$.

We next show that unconditional sums respect subtraction.  (For this, we will not use the first assumption.)  Let $\underline{a}\in X^I$ and $\underline{b}\in X^J$ be families with unconditional sums $x$ and $y$, respectively.  As unconditional sums are unaffected by adjunction or removal of $0$ entries, we may assume $I=J$.  Let $S\in \mathscr{A}$, and fix $T,U\in \mathscr{A}$ with $T-U\in S$.  Fix finite sets $F_{T,\underline{a}}$ and $F_{U,\underline{b}}$ satisfying the conditions in Definition \ref{Def:UncondSums}.  Taking $F:=F_{T,\underline{a}}\cup F_{U,\underline{b}}$, then for each finite set $F'$ satisfying $F\subseteq F'\subseteq I$ we have
\[
\sum_{i\in F'}(a_i-b_i)-(x-y) = \left(\sum_{i\in F'}a_i-x\right)-\left(\sum_{i\in F'}b_i-y\right)\in T-U\subseteq S,
\]
thus verifying that $x-y$ is an unconditional sum for $\underline{a}-\underline{b}$.

Taking $\underline{a}$ to be the empty family, which by Proposition \ref{Prop:UniqueUncondSums} has an unconditional sum of $0$, we see that $-\underline{b}$ has unconditional sum $-y$, and so unconditional sums respect negations.  After replacing $\underline{b}$ by $-\underline{b}$, since sums respect subtraction, they also respect addition.
\end{proof}

When $\mathscr{A}$ is a filter, the two hypotheses of Proposition \ref{Prop:SecondUniqueUncondSums} hold exactly when $\mathscr{A}$ is the set of (not necessarily open) neighborhoods of $0$ for a (unique) topology that makes $(X,+)$ a Hausdorff topological group (see Propositions 1 and 2 of Chapter 3, Section 1.2 of \cite{Bourbaki}).  Unsurprisingly, those two hypotheses are sufficient but not necessary for uniqueness of unconditional sums, as the following example demonstrates.

\begin{example}\label{Example:DeletedNeighborhoods}
Let $X=\R$, with its usual group structure.  Take $\mathscr{A}$ to be the set of open neighborhoods of $0$ in the usual topology on $\R$, but remove $0$ from each neighborhood.  These are commonly called \emph{deleted} (or \emph{punctured}) neighborhoods.

A family $\underline{a}\in X^I$ has an unconditional sum $x\in \R$, relative to $\mathscr{A}$, exactly when $\sum_{i\in I}a_i$ is an absolutely convergent series (in the usual sense, after reindexing by $\N$), with the sum of that series being $x$, and there is a finite subset $F\subseteq I$ that \emph{cannot} be finitely extended to $F'\subseteq I$ with $\sum_{i\in F'}a_i=x$.  (Equivalently, $\underline{a}$ consists of terms of an absolutely converging series, with sum $x$, and under any reindexing of $I$ with $\N$, initial subsums eventually fail to equal $x$.)  Unconditional sums relative to $\mathscr{A}$ are unique when they exist, but neither of the hypotheses in Proposition \ref{Prop:SecondUniqueUncondSums} hold.
\end{example}

The unconditional summation system for the previous example \emph{fails} Axiom \ref{Axiom:SubsSummable}.  This demonstrates another fundamental difference between the methods of partial summation and unconditional summation; the first type of summation requires recursive constructions, and the second is of a more global nature.  Further, this failure undercuts our intuition from (normal) absolutely converging series, where Axiom \ref{Axiom:SubsSummable} does hold.  It turns out that it holds in that case for a special reason.  The following two results make it clear when to expect subfamilies to remain unconditionally summable.

\begin{prop}
Let $(X,+)$ be an abelian group, and let $\mathscr{A}\subseteq \power(X)$.  Assume
\begin{itemize}
\item for each $S\in \mathscr{A}$ there exist $T,U\in \mathscr{A}$ with $T-U\subseteq S$.
\end{itemize}
Let $\underline{a}\in X^I$ be an unconditionally summable family, with sum $x\in X$.  Let $P$ be a partition of $I$, and for each $J\in P$ assume that $(a_i)_{i\in J}$ is summable, with sum $s_J\in X$.  Then $x$ is an  unconditional sum for $\underline{s}:=(s_J)_{J\in P}$.

In particular, if the two assumptions of \textup{Proposition \ref{Prop:SecondUniqueUncondSums}} hold, and if subfamilies of summable families remain summable, then $\Sigma_{\mathscr{A}}$ satisfies \textup{Axiom \ref{Axiom:InsertAssociativity}}.
\end{prop}
\begin{proof}
Let $S\in \mathscr{A}$.  Fix $T,U\in \mathscr{A}$ with $T-U\subseteq S$.  There exists a finite set $F_{T,\underline{a}}\subseteq I$, such that $\sum_{i\in F'}a_i-x\in T$, for each finite set $F'$ satisfying $F_{T,\underline{a}}\subseteq F'\subseteq I$.

Fix $F_{S,\underline{s}}:=\{J\in P\,:\, J\cap F_{T,\underline{a}}\neq \emptyset\}$, which is a finite subset of $P$.  Let $F^{\ast}$ be an arbitrary finite set satisfying $F_{S,\underline{s}}\subseteq F^{\ast}\subseteq P$.

By the proof of Proposition \ref{Prop:SecondUniqueUncondSums}, unconditional sums respect (finite) addition.  So, there is a finite set $F\subseteq I$ containing $F_{T,\underline{a}}$, such that $\sum_{i\in F}a_i - \sum_{J\in F^{\ast}}s_J\in U$.  We then find
\[
\sum_{J\in F^{\ast}}s_J-x = \left(\sum_{i\in F}a_i-x\right)-\left(\sum_{i\in F}a_i-\sum_{J\in F^{\ast}}s_J\right)\in T-U\subseteq S.
\]
The last sentence of the proposition immediately follows.
\end{proof}

\begin{thm}
Let $(X,+)$ be an abelian group, and let $\mathscr{A}\subseteq \power(X)$.  Assume
\begin{itemize}
\item for each $S\in \mathscr{A}$ there exist $T,U\in \mathscr{A}$ with $T-U\subseteq S$.
\end{itemize}
If a sequence $\underline{a}\in X^I$ is unconditionally summable, then it is \emph{sum-Cauchy}, meaning that for each $S\in \mathscr{A}$ there exists a finite set $F\subseteq I$ such that for each finite set $F'\subseteq I$ with $F\cap F'=\emptyset$ it holds that $\sum_{i\in F'}a_i\in S$.

Thus, being sum-Cauchy is necessary for summability.  It is sufficient if and only if subfamilies of summable families remain summable.
\end{thm}
\begin{proof}
Let $S\in \mathscr{A}$ and fix $T,U\in \mathscr{A}$ such that $T-U\subseteq S$.   Let $F_{T,\underline{a}}$ and $F_{U,\underline{a}}$ satisfy the appropriate condition from Definition \ref{Def:UncondSums}.  Take $F:=F_{T,\underline{a}}\cup F_{U,\underline{a}}$, and let $F'\subseteq I$ be any finite set with $F\cap F'=\emptyset$.  We find
\[
\sum_{i\in F'}a_i =\sum_{i\in F\cup F'}a_i - \sum_{i\in F}a_i\in T-U\subseteq S.
\]

For the forward direction of the biconditional, it suffices to observe that the sum-Cauchy property passes to subfamilies; the details are as follows.  Let $\underline{a}\in X^I$ be a sum-Cauchy family.  Fix a finite set $F\subseteq I$ such that for each finite set $F'\subseteq I$ with $F\cap F'=\emptyset$, then $\sum_{i\in F'}a_i\in S$.  Let $J\subseteq I$ be arbitrary.  We can use the finite set $F\cap J\subseteq J$ to verify the sum-Cauchy condition for the subfamily indexed by $J$, since for each finite set $F'\subseteq J$ with $(F\cap J)\cap F'=\emptyset$, we have both $F'\subseteq I$ and  $F\cap F'=\emptyset$, so $\sum_{i\in F'}a_i\in S$.

For the backward direction, we use an argument suggested to me by George Bergman twenty years ago.  Assume that subfamilies of summable families are summable.  Let $\underline{a}\in X^I$ be sum-Cauchy.  Also, let $\varphi\colon I\to J$ be a bijection, where $I\cap J=\emptyset$. Finally, let $\underline{b}\in X^{I\cup J}$ be the family where
\[
b_k=\begin{cases}
a_k & \text{ if $k\in I$},\\
-a_{\varphi^{-1}(k)} & \text{ if $k\in J$}.
\end{cases}
\]
In other words, $\underline{b}$ is the merger of $\underline{a}$ and a disjoint copy of $-\underline{a}$.

We claim that $\underline{b}$ is unconditionally summable, with sum $0$.  To see this, let $S\in \mathscr{A}$, and fix $T,U\in \mathscr{A}$ with $T-U\subseteq S$.  Further, let $F_1\subseteq I$ be a finite subset such that for any finite set $F'\subseteq I$ with $F\cap F'=\emptyset$, then $\sum_{i\in F'}a_i\in T$.  Similarly, let $F_2\subseteq I$ do the same thing, relative to $U$.  Finally, let
\[
F^{\ast}:=(F_1\cup F_2)\cup \varphi(F_1\cup F_2)\subseteq I\cup J.
\]
For each finite set $F'$ with $F^{\ast}\subseteq F'\subseteq I\cup J$, we find
\[
\sum_{k\in F'}b_k-0 = \sum_{k\in (F'\setminus (F_1\cup F_2))\cap I}a_k -\sum_{k\in (F'\setminus \varphi(F_1\cup F_2))\cap J} a_{\varphi^{-1}(k)}\in T-U\subseteq S,
\]
thus verifying the unconditional summability of $\underline{b}$.  Since $\underline{a}$ is a subfamily, it is summable.
\end{proof}

The sum-Cauchy condition is a modification of the usual Cauchy condition for sequences.  When the sum-Cauchy condition is equivalent to unconditional summability, one might say that $\mathscr{A}$ is a \emph{sum-complete} system.  This is a weakening of the notion of a \emph{complete} topology.

We now have natural conditions on $\mathscr{A}$ wherewith the unconditional summation $\Sigma_{\mathscr{A}}$ satisfies the previously numbered axioms.  Next, we change gears to determine from a summation system $\Sigma$ the best possible collection of sets $\mathscr{A}\subseteq \power(X)$ under which we might expect $\Sigma$ to be a subsystem of $\Sigma_{\mathscr{A}}$.  To accomplish this task, we use a result that is parallel with the fundamental theorem of topological limits (i.e., Proposition \ref{Prop:FinestTopologyWithLimitsDetermined}).

\begin{lemma}\label{Lemma:SubsetsUncondReversal}
Let $(X,+)$ be an abelian group, and let $\mathscr{A}\subseteq \mathscr{B}\subseteq \power(X)$.  If $x\in X$ is an unconditional sum for a family $\underline{a}\in X^I$ relative to $\mathscr{B}$, then $x$ is also an unconditional sum for $\underline{a}$ relative to $\mathscr{A}$.

Given any collection of pairs $(\underline{a},x)$, there is a largest $\mathscr{C}\subseteq \power(X)$ such that $\underline{a}$ is summable to $x$ relative to $\mathscr{C}$, for each such pair.  If $\mathscr{C}\neq \power(X)$, then $\mathscr{C}$ is a filter on $X$.
\end{lemma}
\begin{proof}
For each set $S\in \mathscr{A}$, we have $S\in \mathscr{B}$.  The existence of an appropriate finite set $F_{S,\underline{a}}$ is guaranteed by the assumption that $x$ is a $\mathscr{B}$-unconditional sum for $\underline{a}$.

For the second paragraph, let $\mathscr{C}$ consist of all sets $S\subseteq X$ satisfying the needed condition of Definition \ref{Def:UncondSums}, for each given pair $(\underline{a},x)$.  In particular, $X\in \mathscr{C}$.

Given any $S\in \mathscr{C}$, if $S\subseteq S'\subseteq X$, then we see that $S'\in \mathscr{C}$ by simply taking $F_{S',\underline{a}}$ to be $F_{S,\underline{a}}$ (for each family $\underline{a}$ appearing as a first coordinate).  Thus, $\mathscr{C}$ is upwards closed.

Next, given $S,T\in \mathscr{C}$, we see that $S\cap T\in \mathscr{C}$ by simply taking $F_{S\cap T,\underline{a}}$ to be $F_{S,\underline{a}}\cup F_{T,\underline{a}}$.

Finally, if $\emptyset \in \mathscr{C}$, then there are no unconditional sums, and so $\mathscr{C}=\power(X)$.  Contrapositively, if $\mathscr{C}\neq \power(X)$, then $\emptyset\notin \mathscr{C}$, and so $\mathscr{C}$ is a filter.
\end{proof}

One cannot strengthen the first conclusion of Lemma \ref{Lemma:SubsetsUncondReversal} to claim that $\Sigma_{\mathscr{A}}\supseteq \Sigma_{\mathscr{B}}$, because a family may have multiple unconditional sums relative to $\mathscr{A}$ but a unique unconditional sum relative to $\mathscr{B}$; examples are easy to construct.  However, that stronger conclusion does hold if all unconditional sums relative to $\mathscr{A}$ are unique.

We are now equipped to describe the ``best'' filter under which the sums in a summation system $\Sigma$ are also unconditional sums (even if not necessarily uniquely so).

\begin{definition}
Given an abelian group $(X,+)$ with a summation system $\Sigma$, then the filter described in Lemma \ref{Lemma:SubsetsUncondReversal} for the ordered pairs in $\Sigma$ is the {\bf $\Sigma$-filter}, denoted $\mathscr{F}_{\Sigma}$.
\end{definition}

As mentioned previously, unconditional summability is usually defined by taking $\mathscr{A}$ to be the collection of neighborhoods of $0$ under some topology on $X$.  Generally, from the set of unconditional sums $\Sigma_{\mathscr{A}}$ we can only recover the filter $\mathscr{F}_{\Sigma_{\mathscr{A}}}$, which contains those neighborhoods of zero (possibly properly).  Since $\mathscr{F}_{\Sigma_{\mathscr{A}}}\cup\{\emptyset\}$ is itself a topology, we do recover a topology, but in some sense it is ``concentrated'' around $0$.

There is an order reversing relationship between summation systems and their corresponding filters.

\begin{lemma}\label{Lemma:OrderReverseFilter}
Let $(X,+)$ be an abelian group.  If $\Sigma\subseteq \Sigma'$ are summation systems on $X$, then $\mathscr{F}_{\Sigma}\supseteq \mathscr{F}_{\Sigma'}$.
\end{lemma}
\begin{proof}
This follows from the fact that $\mathscr{F}_{\Sigma}$ (respectively, $\mathscr{F}_{\Sigma'}$) is the largest subset $\mathscr{C}\subseteq \power(X)$ where the sums in $\Sigma$ (respectively, $\Sigma'$) are unconditional sums relative to $\mathscr{C}$.
\end{proof}

Lemmas \ref{Lemma:SubsetsUncondReversal} and \ref{Lemma:OrderReverseFilter} give us a ``near'' Galois connection between subsets of $\power(X)$ and summation systems on $X$.  This induces a map that acts almost like a closure operator, just as in Theorem \ref{Thm:BigTopTheorem}.  Its relevant properties are as follows.

\begin{thm}
Let $(X,+)$ be an abelian group.  If $\psi\colon \power(\power(X))\to \power(\power(X))$ is the map defined by the rule $\mathscr{A}\mapsto \mathscr{F}:=\mathscr{F}_{\Sigma_{\mathscr{A}}}$, then
\begin{itemize}
\item[\textup{(1)}] $\psi$ is extensive and idempotent,
\item[\textup{(2)}] $\Sigma_{\mathscr{A}}\subseteq \Sigma_{\psi(\mathscr{A})}$ for each $\mathscr{A}\subseteq \power(X)$, and
\item[\textup{(3)}] when restricted to those elements $\mathscr{A}\in \power(\power(X))$ whose unconditional sums are unique, $\psi$ is a closure operator and $\Sigma_{\mathscr{A}}= \Sigma_{\psi(\mathscr{A})}$.
\end{itemize}
\end{thm}
\begin{proof}
Throughout the proof, for notational convenience let $\mathscr{A}\subseteq \power(X)$ be arbitrary, take $\Sigma:=\Sigma_{\mathscr{A}}$, and take $\Sigma':=\Sigma_{\psi(\mathscr{A})}$.

{\bf (1) and (2)}:  Note that, by definition of $\Sigma$, the sums under $\Sigma$ are unconditional sums relative to $\mathscr{A}$.  We thus have $\mathscr{A}\subseteq \psi(\mathscr{A})$, since $\psi(\mathscr{A})$ is the largest $\mathscr{C}\subseteq \power(X)$ such that the sums in $\Sigma$ are unconditional sums relative to $\mathscr{C}$.  This proves the extensive property.

Letting $(\underline{a},x)\in \Sigma$,  then $\underline{a}$ has $x$ as its \emph{unique} unconditional sum relative to $\mathscr{A}$, and so the same is true for the larger system $\psi(\mathscr{A})$ by Lemma \ref{Lemma:SubsetsUncondReversal}, noting that $\underline{a}$ continues to have at least one unconditional sum relative to $\psi(\mathscr{A})$.  Thus $\Sigma\subseteq \Sigma'$.

Applying Lemma \ref{Lemma:OrderReverseFilter}, we have $\psi(\mathscr{A})\supseteq \psi^2(\mathscr{A})$.  The reverse holds from the extensive property, and so $\psi$ is idempotent.

{\bf (3)}:  Assume $\mathscr{A}\subseteq \mathscr{B}\subseteq \power(X)$, and also assume that the unconditional sums relative to $\mathscr{A}$ are unique when they exist.  Then $\Sigma_{\mathscr{A}}\supseteq \Sigma_{\mathscr{B}}$, as noted in the paragraph following the proof of Lemma \ref{Lemma:SubsetsUncondReversal}.  Applying Lemma \ref{Lemma:OrderReverseFilter} once again, we see that $\psi$ is monotonic, and hence it is a closure operator.

Taking $\mathscr{B}:=\psi(\mathscr{A})$ in the previous paragraph yields $\Sigma'\subseteq \Sigma$, and the reverse inclusion was already demonstrated.
\end{proof}

Theorem \ref{Thm:BigTopTheorem}(3) says that all topologies in the image of $\varphi$ are $T_1$.  Proposition \ref{Prop:UniqueUncondSums} gives us a rough analogue of this fact for $\psi$, at least when finite sums behave the way we expect.   There is no analogue for Theorem \ref{Thm:BigTopTheorem}(1), as the following example shows that $\psi$ does depend on the additive structure imposed on $X$.

\begin{example}
Let $X=\Q$, with its usual additive structure.  For each $n\in \N$, let $S_n$ be the interval $(-1/n,1/n)\cap \Q$ (with $S_0=\Q$), and let $\mathscr{A}$ be the collection of these intervals.  Note that $(1,1/2,1/4,\ldots)$ is unconditionally summable, with unique sum $2$, relative to $\mathscr{A}$.  Consequently, $\mathscr{F}_{\Sigma_{\mathscr{A}}}$ is \emph{not} the principal ultrafilter generated by $0$.

For each $n\in \N$, let $T_n=\{0\}\cup\{\pm 5^{m2^n}\, :\, m\in \Z_{\geq 1}\}$ (with $T_0=\Q$), and let $\mathscr{B}$ be the collection of these sets.  It is not hard to show that the unconditionally summable families relative to $\mathscr{B}$ (or even just $T_1$) are exactly the finite families up to core-extensions.  Thus, $\mathscr{F}_{\Sigma_{\mathscr{B}}}$ is the principal ultrafilter generated by $0$.

The set $\Q$ decomposes as the disjoint union
\[
\{0\}\cup (S_0\setminus S_1) \cup (S_1\setminus S_2)\cup \cdots
\]
where each piece is countably infinite except the first.  There is a similar decomposition in terms of the $T_i$.  Thus, there exists a bijection $\Q\to\Q$ sending $0\mapsto 0$ and that restricts to a bijection from $S_n$ to $T_n$ for each $n\in \N$.  (This map cannot be group homomorphism, but it can be chosen to respect negation.)  This bijection allows us to identify $\mathscr{B}$ with $\mathscr{A}$, but with $\Q$ given a different additive structure.  We see that $\psi(\mathscr{A})$ depends on which of those two additive structures on $X$ is used.
\end{example}

Since a topology $\tau$ and its filter $\mathscr{F}$ of neighborhoods of $0$ both give rise to the same type of unconditional summability $\Sigma$, and since $\tau':=\mathscr{F}\cup\{\emptyset\}$ is a topology in its own right, the most one could hope to expect to recover from $\Sigma$ is $\mathscr{F}$ (in the case when $\mathscr{F}$ is the $\Sigma$-filter).  The following example illustrates subtleties in this regard.

\begin{example}\label{Example:RegeneratingRTopTwoWays}
Let $\Sigma$ be the summation system for the (standard) absolutely convergent series on $X=\R$.  It turns out that the $\Sigma$-topology (as in Definition \ref{Def:SigmaTopology}) is the usual topology $\tau$ on $\R$, which is quite surprising because $\Sigma$ is \emph{not} the induced summation system from $\tau$, but rather a small subsystem (after removing any non-ordinal-indexed sums from $\Sigma$).  This claim follows from the computations done in Example \ref{Example:Rchain}, noting the final parenthetical sentence given there.

Thus, one can recover the original topology on $\R$ from the absolutely convergent series system $\Sigma$.  However, it is recovered via the $\Sigma$-topology route.  If, instead, we used the $\Sigma$-filter route, the best we could do is recover the filter of neighborhoods of zero.  To finish this example, let us show that this is exactly the filter we recover.

Suppose, by way of contradiction, there is some $S\in \mathscr{F}_{\Sigma}$ that is not a (standard) neighborhood of $0$.  We can then fix elements $x_0,x_1,\ldots$ not in $S$ with $|x_i|<(i+1)^{-2}$.  Similar to what happened in Example \ref{Example:Rchain}, the family
\[
\underline{a}=(x_0-0, x_1-x_0,x_2-x_1,\ldots)
\]
is absolutely summable, but this time with sum $0$.  Since $S$ is a filter element, there is some finite set $F_{S,\underline{a}}\subseteq \N$ with $\Sigma (a_i)_{i\in F'}-\Sigma\underline{a}\in S$ for each finite set $F'$ satisfying $F_{S,\underline{a}}\subseteq F'\subseteq \N$.  Thus, for some sufficiently large integer $N$ it holds that $x_N=\sum_{i=0}^{N}a_i-0\in S$, giving the needed contradiction.
\end{example}

The previous example suggests that when attempting to recover a topology, it may accidentally turn out better to use the $\Sigma$-topology route, when available, instead of the $\Sigma$-filter route, even when a summation system arises from unconditional summation methods.  The following result partly explains such accidents.

\begin{thm}
Let $(X,+)$ be an abelian group with a summation system $\Sigma$ satisfying
\begin{itemize}
\item reindexing invariance, \textup{Axiom \ref{Axiom:ReindexInvariance}},
\item empty sum existence, \textup{Axiom \ref{Axiom:EmptyExists}}, with $\Sigma()=0$,
\item insertive associativity, \textup{Axiom \ref{Axiom:InsertAssociativity}}, and
\item additive extension closure, \textup{Axiom \ref{Axiom:AdditiveExtClosure}}.
\end{itemize}
In this situation, the $\Sigma$-filter contains the neighborhoods of $0$ in the $\Sigma$-topology.
\end{thm}
\begin{proof}
Let $U\in \tau_{\Sigma}$ with $0\in U$.  Suppose, by way of contradiction, that $U\notin \mathscr{F}_{\Sigma}$.  Fix some family $\underline{a}\in X^I\cap \dom(\Sigma)$ such that for each finite subset $F\subseteq I$ there is another finite subset $F'$ satisfying $F\subseteq F'\subseteq I$ and $\Sigma (a_i)_{i\in F'}-\Sigma \underline{a}\notin U$.  (We implicitly use the second and fourth bullet points here, so that finite sums under $\Sigma$ are just iterated additions.)  Note that $I$ must be infinite, else taking $F=I$, then $F'=I$, which would force $0=\Sigma (a_i)_{i\in F'}-\Sigma \underline{a}\notin U$.

By the fourth bullet point, we may as well assume $\Sigma \underline{a}=0$ (after replacing $\underline{a}$ with the new family where $-\Sigma\underline{a}$ is adjoined to $\underline{a}$, and expanding $F'$ to always include that one new index).

Our new goal is to show that for any infinite subset $J\subseteq I$, there is another set $J'$ satisfying $J\subseteq J'\subseteq I$, with $|J|=|J'|$ and $\Sigma(a_i)_{i\in J'}\notin U$.  This will finish the proof because taking $J=I$, then $J'=I$, hence $0=\Sigma(a_i)_{i\in J'}\notin U$, yielding the needed contradiction.

First, consider the case when $|J|=\aleph_0$.  List the elements of $J$ as $j_0,j_1,\ldots$.  Set $J_0:=\{j_0\}$.  For each $n\in \N$, if a finite set $J_n$ is given, recursively define $J_n'\subseteq I$ to be a finite set containing $J_n$ where $\Sigma(a_i)_{i\in J_n'}\notin U$, and put
\[
J_{n+1}:=J_n'\cup\{j_i\, :\, i\in \N\text{ is the least index such that $j_i\notin J_n'$}\}.
\]  Letting $J':=\bigcup_{n\in \N}J'_n$, we see that $|J'|=\aleph_0=|J|$ with $J\subseteq J'\subseteq I$.

Well-order $J'$, with order type $\omega$, such that the elements of $J_{n}'$ occur before those in $J_{n+1}'\setminus J_{n}'$, for each integer $n\geq 0$.  Let $s_n:=\Sigma (a_i)_{i\in J_{n}'}$ for each $n\in \N$.  After using reindexing invariance to replace $J'$ with $\omega$, then $\underline{s}=(s_0,s_1,\ldots)$ is a cofinal subsequence of the partial sums of $(a_i)_{i\in J'}$, under that $\omega$-ordering of $J'$.

For use shortly, note that insertive associativity implies Axiom \ref{Axiom:InitialSegments}.  Thus, ordinal-indexed $\Sigma$-summable families always have partial sum families.  Further, since $\Sigma$ satisfies postfix associativity (by the third and fourth bullet points) as well as ordinal insertive associativity (by the first and third bullet points), it also satisfies \eqref{Eq:SubsequenceProperty} by Lemma \ref{Lemma:SubsequenceLimits}.  Hence $\underline{s}$ has a $\Sigma$-limit, which is $\Sigma (a_i)_{i\in J'}$.  Since $U$ is open, with $s_n\notin U$ for each $n\in \N$, then $\Sigma (a_i)_{i\in J'}\notin U$.

Now consider the case when $|J|>\aleph_0$.  By transfinite induction, we may assume that the claim is true for all infinite subsets of $I$ of cardinality smaller than $J$.  Index the elements of $J$ by $\alpha$, where $\alpha$ is the first ordinal of cardinality $|J|$.  For each $\beta\in \alpha$, we will recursively define sets $J_{\beta}$ and $J_{\beta}'$ satisfying:
\begin{itemize}
\item[(1)] $J_{\beta}\subseteq J_{\beta}'\subseteq I$,
\item[(2)] $|J_{\beta}'|=|J_{\beta}|=\max(\aleph_0,|\beta|)$,
\item[(3)] $K_{\beta}:=\bigcup_{\gamma\in \beta}J_{\gamma}'\subsetneq J_{\beta}$
\item[(4)] $j_{\beta}\in J_{\beta}$, and
\item[(5)] $\Sigma (a_i)_{i\in J_{\beta}'}\notin U$.
\end{itemize}
To start the recursion, let $J_{0}=\{j_i\, :\, i<\omega\}$.  Note that (3) and (4), as well as the second equality in (2), all hold when $\beta=0$.  If $J_{\beta}$ is defined for some $\beta\in \alpha$, use the inductive assumption to find a set $J_{\beta}'$ satisfying (1), the first equality in (2), and (5).  Finally, if $J_{\gamma}'$ is defined for each $\gamma\in \beta$, when $\beta>0$, then putting
\[
J_{\beta}:=K_{\beta}\cup \{j_{\beta}\}\cup \{j_i\, :\, \text{$i\in \alpha$ is the smallest index with $j_i\notin K_{\beta}$}\}
\]
we find that (3) and (4) hold.  The second equality in (2) also holds when $\beta>0$, as follows.  First, $|J_{\beta}|$ is at least $\aleph_0$, since $J_{\beta}\supseteq J_0'\supseteq J_0$.  Second, $|J_{\beta}|\geq |\beta|$, since $J_{\beta}\supseteq \{j_i\, :\, i<\beta\}$.  Finally, from the definition of $J_{\beta}$ we find
\[
|J_{\beta}| \leq 2 + \sum_{\gamma<\beta}\max(\aleph_0,|\gamma|) \leq \max(\aleph_0,|\beta|).
\]

An argument similar to the one given in the case when $|J|=\aleph_0$ shows that $J':=\bigcup_{\beta\in \alpha}J_{\beta}'$ satisfies the needed conditions to finish the proof.  (As before, well-order the set $J'$ such that all ``new elements'' $J_{\beta}'\setminus K_{\beta}\neq \emptyset$ occur after the ``old elements'' in $K_{\beta}$.  The only slightly tricky issue is to make sure that these new elements are given a successor order type, to guarantee that the sum over $J_{\beta}'$ is a partial sum, by Lemma \ref{Lemma:SuccessorLimits}.)
\end{proof}

Even more can be said when $X$ is a topological group under a first countable topology, as is the case in Example \ref{Example:RegeneratingRTopTwoWays}.  We leave such investigations to the reader.

For endomorphism rings there is an even tighter connection between the method of partial sums and the method of unconditional summing.  Let $R$ be a ring, let $X=\End(M_R)$, where $M$ is a right $R$-module, and finally let $\Sigma$ be the endomorphism ring summation system on $X$ from Example \ref{Example:Endos}.  This type of summation can be handled using topology, as first done by Szele in \cite{Szele}, where a basis of open neighborhoods is given by translations of annihilators of finite subsets of $M$. The interested reader is invited to show that this is a topology that makes $X$ a Hausdorff topological ring, under the usual ring operations on $X$; it is commonly called the \emph{finite topology}.

Continuing the notation from the previous paragraph, it is well-known that $\Sigma$-summation is exactly unconditional summation in the finite topology, using the standard group structure on $X$.  The proof is almost definitional.  This topology makes $X$ a complete topological group, and hence sum-complete.  As the finite topology has a basis of neighborhoods of $0$ consisting of left ideals, the sum-Cauchy property for a family $\underline{a}\in X^I$ is equivalent to the (generally weaker) condition that for each $S\in \mathscr{A}$, then $a_i\in S$ for all but finitely many $i\in I$.

Surprisingly, under the finite topology the unconditional sums are nothing more than reindexed sums coming from partial summation, as we now show.

\begin{prop}
Let $R$ be a ring, let $M_R$ be a right $R$-module, let $X=\End(M_R)$ with its usual ring structure, and let $\Sigma$ be the summation system of \textup{Example \ref{Example:Endos}}.  If $\Sigma_1$ is the restriction of $\Sigma$ to ordinal-indexed families, and if $\Sigma_2$ is the induced summation system under the finite topology on $X$, then $\Sigma_1=\Sigma_2$.
\end{prop}
\begin{proof}
Let $\underline{a}\in X^I$ with $I\subseteq {\rm Ord}$.  If $I=\emptyset$, then both types of sums exists and equal $0$.  Thus, we reduce to the case when $I\neq \emptyset$.  By transfinite induction, assume for all families whose index set is of order type smaller than $\sigma(I)$ that the two types of sums match.

{\bf Case 1}: Suppose, by way of contradiction, that $\underline{a}\in \dom(\Sigma_2)\setminus \dom(\Sigma_1)$.  There then exists some $m\in M$ such that $a_i(m)\neq 0$ for infinitely many $i\in I$.  Fix $i_0<i_1<\ldots$ to be the first $\omega$ such indices, and fix
\[
\beta=\inf\{j\in {\rm Ord}\, :\, j>i_n\text{ for each }n\in \N\}.
\]
Note that $(a_i)_{i\in I_{<\beta}}$ is not $\Sigma_1$-summable, but it is $\Sigma_2$-summable.  Thus, our inductive assumption forces the indices $i_0<i_1<\ldots$ to be cofinal in $I$.  Taking limits in the finite topology, since $\Sigma_2\underline{a}$ exists we find
\[
\Sigma_2\underline{a} = \lim_{i\in I}\Sigma_2(a_j)_{j\in I_{\leq i}}=\lim_{i\in I}\Sigma_1 (a_j)_{j\in I_{\leq i}},
\]
where the first equality comes from Lemma \ref{Lemma:SuccessorLimits} (noting that $\Sigma_2$ satisfies postfix associativity by Theorem \ref{Thm:WhatPartialSummationGetsUs}(4)), and the second equality comes from the inductive assumption.  Thus, $\Sigma_1 (a_j)_{j\in I_{\leq i}}-\Sigma_2\underline{a}\in \ann(m)$ for all sufficiently large indices $i\in I$.  Hence, for all sufficiently large $n\in \N$, we find
\[
\Sigma_1 (a_j)_{j\in I\, :\, i_n<j\leq i_{n+1}} = \Bigl( \Sigma_1 (a_{j})_{j\in I_{\leq i_{n+1}}}-\Sigma_2\underline{a}\Bigr) - \Bigl( \Sigma_1 (a_j)_{j\in I_{\leq i_{n}}}-\Sigma_2\underline{a}\Bigr)\in \ann(m).
\]
However, this cannot be the case because there is exactly one $j\in I$ with $i_n<j\leq i_{n+1}$ where $a_j(m)\neq 0$.

{\bf Case 2}: Suppose, by way of contradiction, that $\underline{a}\in \dom(\Sigma_1)\setminus \dom(\Sigma_2)$.  By our inductive assumption $\underline{s}:=p(\underline{a})$ must exist, and the partial sums agree under $\Sigma_1$ and $\Sigma_2$.  Yet, $\underline{s}$ has no topological limit (else that limit would be unique, and then $\underline{s}\in \dom(\Sigma_2)$). So the order type of $I$ is a limit ordinal, and further there are some elements $m_1,m_2,\ldots, m_{\ell}\in M$ and some cofinal subset $K\subseteq I$ such that $s_k-\Sigma_1 \underline{a}\notin \ann(m_1,m_2,\ldots, m_{\ell})$ for each $k\in K$.  Since, $\sigma(I)$ is a limit ordinal, we can find a single $m\in \{m_1,\ldots, m_{\ell}\}$ such that $s_k-\Sigma_1 \underline{a}\notin \ann(m)$ for each $k\in K$, after passing to a subsequence of $K$ if necessary.  However, there are only finitely many indices $i\in I$ such that $a_i(m)\neq 0$.  Thus, $s_k$ and $\Sigma_1 \underline{a}$ agree on $m$, for each $k$ past that last index, yielding the needed contradiction.

{\bf Case 3}: Assume that $\underline{a}\in \dom(\Sigma_1)\cap \dom(\Sigma_2)$.  It suffices to show that the two sums are the same.  Note that the family of partial sums $\underline{s}=(s_i)_{i\in I}$ exists.  For each $m\in M$, there is a nonempty finite subset $F_m\subseteq I$ such that $a_i(m)\neq 0$ only if $i\in F_m$.  Fixing $\alpha_m=\max(F_m)$, then for each $\beta\in I_{\geq \alpha_m}$
\[
\left(\Sigma_1 \underline{a}\right)(m) = \left(\Sigma_1 (a_i)_{i\in F_m}\right)(m) = s_{\beta}(m).
\]
Thus,
\[
\Sigma_2\underline{a}:=\lim_{\beta\in I}s_{\beta}=\Sigma_1 \underline{a}
\]
since $s_{\beta}-\Sigma_1\underline{a}\in \ann(m)$ for each $m\in M$, for all sufficiently large $\beta$ (depending on $m$).

As there is nothing to show in the fourth case, this finishes the proof.
\end{proof}

The finite topology $\tau$ is Hausdorff, and so the last sentence of Corollary \ref{Cor:InducedSelfMapSpecial} tells us that its induced summation system is unchanged if we use the topology $\varphi(\tau)$ instead.  Putting this another way, endomorphism ring summation $\Sigma$ is maximally defined with respect to (reindexed) partial sums.  Any strictly finer topology than $\varphi(\tau)$ has \emph{fewer} gapless limits, and therefore also fewer partial sums.

This is not to say that the finite topology is always fixed under the induced self-map $\varphi$.  To give one example, take $R$ to be a nonzero ring, take $M_R$ to be an $\aleph_1$-dimensional free $R$-module, say with basis $(b_{\alpha})_{\alpha<\omega_1}$, and fix $X:=\End(M_R)$.  For each ordinal $\alpha<\omega_1$, define an endomorphism $x_{\alpha}\in X$ by first defining it on the basis by the rule
\[
x_{\alpha}(b_{\beta})=\begin{cases}
0 & \text{ if $\beta\leq \alpha$,}\\
b_{\alpha} & \text{ otherwise}
\end{cases}
\]
and then extending to all of $M$ linearly.  Finally, take $S:=\{x_{\alpha}\, :\, \alpha<\omega_1\}$.

The set $S$ is not closed in the finite topology, since the zero endomorphism is the limit of the $x_{\alpha}$, in the given order.  However, $S$ is closed under gapless limits.  To see this, it suffices to show that any sequence $\underline{s}$ that does not consist of finitely many constant strings has no limit.  After passing to a subsequence, we may as well assume $\underline{s}=(s_i)_{i<\omega}$ and $s_{i}\neq s_{i+1}$ for each $i\in \omega$.  Let
\[
\beta:=\inf\left\{\alpha+1\in {\rm Ord}\, :\, x_{\alpha}=s_i\ \text{ for some $i\in \N$}\right\},
\]
and note that $\beta<\omega_1$.  Now, the $s_i$ do not eventually agree on $b_{\beta}$, and so they do not converge in the finite topology.

This raises the following interesting question:

\begin{question}
If $\Sigma$ is endomorphism ring summation, is there a simple description for $\tau_{\Sigma}$ or for $\mathscr{F}_{\Sigma}$?
\end{question}

\section{Distributivity}\label{Section:Distributivity}

Let $X$ be a set with a summation system $\Sigma$.  Assume $X$ has a multiplication operation, written using concatenation.  The only ring-theoretic axiom connecting addition to multiplication is the distributivity axiom.  Infinitizing this rule is painless.

\begin{axiom}[\bf Infinite distributivity\rm ]\label{Axiom:Distributive}
If $(a_i)_{i\in I}$ and $(b_j)_{j\in J}$ are summable, then so is $(a_ib_j)_{(i,j)\in I\times J}$ with
\[
\left[\Sigma (a_i)_{i\in I}\right]\left[ \Sigma(b_j)_{j\in J}\right]=\Sigma(a_ib_j)_{(i,j)\in I\times J}.
\]
\end{axiom}

In the special case when $|I|=1$, and assuming that singletons sum simply, this infinite distributivity axiom becomes a scalar functoriality axiom.

We end this short section by using distributivity to prove a maximality statement about unconditional summability for series.  Let $X=\R$, with its usual topology, and let $\Sigma_{\rm max}$ be the summation system on $X$ described by the unconditional sums of Definition \ref{Def:UncondSums}.  This system satisfies all of the previously numbered axioms, including the new distributivity axiom (under the usual multiplication in $\R$).

\begin{prop}\label{Prop:AbsoluteSeriesCategorical}
There is no summation system on $\R$ properly extending $\Sigma_{\rm max}$ and satisfying all of the previously numbered axioms.
\end{prop}
\begin{proof}
Suppose, by way of contradiction, that $\Sigma$ is such a proper extension.  Fix
\[
\underline{a}\in \dom(\Sigma)\setminus\dom(\Sigma_{\rm max}).
\]
After passing to a subfamily if necessary, we may also assume that the entries of $\underline{a}$ are all positive or all negative.  From negation functoriality, we reduce to the situation that they are all positive.

There are two main cases to consider.  First, suppose that $\underline{a}$ is indexed by an uncountable family.  Since $\Q_{>0}$ is countable, and since each entry in $\underline{a}$ is bounded below by some element of $\Q_{>0}$, there is some $q\in \Q_{>0}$ such that there are uncountable many entries in $\underline{a}$ all bigger than $q$.  Hence, there is a countable subfamily of $\underline{a}$ with a positive lower bound on its members, and hence not in $\dom(\Sigma_{\rm max})$.

Thus, we are reduced to considering the case when $\underline{a}$ is indexed by $\N$.  As $\underline{a}$ is not absolutely summable, its partial sums diverge to infinity.  Thus, insertive associativity allows us to insert parentheses into $\underline{a}$, around finite sums, and get a new summable family $\underline{b}=(b_i)_{i\in \N}$, with $2b_n<b_{n+1}$ for each $n\in \N$.  (Each summand is still positive.)  The family $\underline{c}:=(1/b_i)_{i\in \N}$ is absolutely convergent, by the ratio test.  By distributivity, the new family $(b_i/b_j)_{(i,j)\in \N\times \N}$ is summable.  This family has infinitely many instances of $1$, contradicting Theorem \ref{Thm:WeakSwindle}.
\end{proof}

Thinking of the values of the Riemann $\zeta$-function as infinite sums in a summation system, it may be interesting to explore which axioms hold, and whether there is some maximality property present.

\section{Endomorphisms and rings}\label{Section:EndosReorderable}

As in the previous section, throughout this section $X$ will be a set with a summation system $\Sigma$ and with a multiplication.

Besides series summation on $\R$, another system involving multiplication is endomorphism summation, as in Example \ref{Example:Endos}.  In unpublished work, Corner \cite{Corner} generalized such summability from a topological framework to a categorical context.  This paper partly grew out of a desire to generalize these ideas further.  Corner listed many standard properties of such sums, but one indispensable feature, appearing as \cite[Lemma 2.2]{Corner}, distinguishes endomorphism summation from most other summation systems we have encountered.  That key feature is the following axiom.

\begin{axiom}[\bf Left multiple summable\rm ]\label{Axiom:LeftMultSum}
If $(a_i)_{i\in I}\in X^I$ is summable, then $(r_ia_i)_{i\in I}$ is summable, for any choice of elements $r_i\in X$.
\end{axiom}

Assuming $X$ is a ring (and identifying $X$ with its $\{0\}$-indexed families), then this axiom implies that $\dom(\Sigma)\cap X$ is a left absorptive subset of $X$.  Moreover, Axiom \ref{Axiom:LeftMultSum} can be viewed as a form of left ideal functoriality (at least, in the presence of the other functoriality axioms).

All previous axioms have been left-right symmetric, so it may feel unnatural to assume left multiples are summable without also assuming right multiples are summable.  However, as the following example demonstrates, this is the norm.  In other words, noncommutativity is inherent in endomorphism ring multiplication, except in special cases.

\begin{example}
Let $V$ be an infinite-dimensional right vector space.  Let $(v_i)_{i\in \N}$ be a family of linearly independent vectors, and let $W$ be a complement to its span.  For each pair of numbers $i,j\in \N$, let $e_{i,j}$ be the unique linear map that sends $v_j$ to $v_i$, but annihilates ${\rm Span}(\{v_k\}_{k\in \N\setminus \{j\}})+W$.  The family $(e_{i,i})_{i\in \N}$ is summable, but $(e_{i,i} e_{i,0})_{i\in \N}$ is not summable since each member is nonzero on $v_0$; right multiple summability fails.
\end{example}

In Axiom \ref{Axiom:SubsSummable} we are not told the exact values for sums over subfamilies.  However, insertive associativity (as well as its monoid merger version) does place some important limits on how subfamilies sum. Likewise, in the statement of Axiom \ref{Axiom:LeftMultSum} we are not told the exact values for sums over left multiple families.  Yet, the following reordering axiom (similar in spirit to \cite[Lemma 2.3]{Corner}) places limitations on how left multiples can sum.  To simplify notation in the following axiom (and elsewhere), given a map $\psi\colon K\to \power(I)$, define the \emph{dual} map $\psi'\colon I\to \power(K)$ by the rule $i\mapsto \{k\in K\, :\, i\in \psi(k)\}$.

\begin{axiom}[\bf Left reorderability\rm ]\label{Axiom:LeftReorder}
If $(a_i)_{i\in I}\in X^I$ is summable, if $(r_k)_{k\in K}\in X^K$ is an arbitrary family, and if $\psi\colon K\to \power(I)$ is a map such that $(r_k)_{k\in \psi'(i)}$ is summable for each $i\in I$, then
\begin{equation}\label{Eq:MainLeftReorder}
\Sigma\left(\left[\Sigma(r_k)_{k\in \psi'(i)}\right]a_i\right)_{i\in I}= \Sigma\left(r_k\left[\Sigma(a_i)_{i\in \psi(k)}\right]\right)_{k\in K}
\end{equation}
where all sums are defined.
\end{axiom}

Aside from the underlying assumption of summability of the family $(a_i)_{i\in I}$, there are four types of summations that occur in \eqref{Eq:MainLeftReorder}.  The first type is the inner sum on the left side; for each $i\in I$ this is the sum $\Sigma(r_k)_{k\in \psi'(i)}$.  It is the main hypothesis required on the function $\psi$ in Axiom \ref{Axiom:LeftReorder} that such sums are defined.  The other three types of summations in \eqref{Eq:MainLeftReorder} are the outer summation on the left side and both the inner and outer sums on the right side.  The fact that these three other types of summations are actually defined is an explicit conclusion of the axiom.

Quite surprisingly, this single axiom---about rearranging terms in summations---captures all other numbered axioms, under only minor hypotheses.  To prove this fact, we first introduce some simplifying notation.  Given $n\in \N$, if
\[
I=\{i\in \N\, :\, i<n\}=\N_{<n},
\]
then we will write the family $(x_i)_{i\in I}$ as $(x_0,x_1,\ldots, x_{n-1})$, leaving the index set as implicitly defined by the notation.

\begin{thm}\label{Thm:MainBigLeftReord}
Let $(X,\cdot, 1)$ be a multiplicative monoid with a summation system $\Sigma$ where
\begin{itemize}
\item $\Sigma$ satisfies left reorderability,
\item $\Sigma$ is surjective,
\item every singleton family whose sole member is $1$ sums to $1$, and
\item there is an element $-1\in X$ such that $\Sigma(1,-1)=\Sigma()$,
\end{itemize}
then $\Sigma$ satisfies all previously numbered axioms, with $+$ the induced addition.  Moreover, the set $X$ is a ring under the induced addition and the given multiplication.
\end{thm}
\begin{proof}
We first show that the monoid merger axiom holds.  Let $\underline{a}\in X^I$ be a summable family, let $K$ be a set, and let $\psi\colon K\to \power(I)$ be a map where each element of $I$ belongs to the image of exactly one element of $K$.  For each $k\in K$ take $r_k=1$.  To apply left reorderability, we need to know that $(1)_{k\in \psi'(i)}$ is summable for each $i\in I$.  From the defining condition on $\psi$, we know that $\psi'(i)$ is always a singleton, and so the family is summable by the third bulleted hypothesis, and it sums to $1$.  (All later uses of left reorderability will, likewise, be made only when its hypotheses are met.  Checking this will be left to the reader.)  Thus, \eqref{Eq:MainLeftReorder} becomes
\[
\Sigma\underline{a}=\Sigma\left(1\Sigma(a_i)_{i\in \psi(k)}\right)_{k\in K}.
\]
On the right side, after absorbing $1$ this is exactly the needed equality to show Axiom \ref{Axiom:PreAbelian}.

By Theorem \ref{Thm:FiniteExtensions}, and using the second bulleted hypothesis, we get reindexing invariance, subfamilies of summable families are summable, $0:=\Sigma()$ exists and can be adjoined or removed from families freely, singletons sum simply, and insertive associativity hold.  As usual, write $+$ for the induced addition (which we do not yet know is total).

Our next goal is to show that $(1,1)$ is a summable family.  Taking $\underline{a}=\underline{r}=(1,-1)$, so $I=K=\{0,1\}$, and taking $\psi\colon K\to \power(I)$ to be the map where $k\mapsto \{k\}$ for each $k\in K$, left reorderability yields that $(1,-1\cdot -1)$ is summable.  (The fact that singletons sum simply is used to check the hypothesis of left reorderability.)  Therefore, it suffices to prove that $-1\cdot -1=1$.

Towards that end, let $x\in X$ be arbitrary.  Take $\underline{a}=(x)$, take $\underline{r}=()$, and take $\psi$ to be the empty map.  Left reorderability gives $0x=0$.  On the other hand, taking $\underline{a}=(x)$, taking $\underline{r}=(1,-1)$, and taking $\psi$ to be the map where $0,1\mapsto \{0\}$, then left reorderability gives $0x=x+-1x$.  Thus, every element $x\in X$ has an additive inverse, namely $-1x$.  In particular, $-1\cdot -1$ is an additive inverse of $-1$, as is $1$.  This does not yet prove that $-1\cdot -1=1$, since the usual proof of the uniqueness of additive inverses uses totality and associativity of addition.  However, we can instantiate the steps in such a proof, as follows.

Take $\underline{a}=(1,-1)$, take $\underline{r}=(1,-1,-1)$, and define $\psi\colon K\to \power(I)$ by the rule that $0,1\mapsto \{0\}$ and $2\mapsto \{1\}$.  We find
\[
-1\cdot -1=0+-1\cdot -1=[1+-1]1+-1\cdot -1=\Sigma(1\cdot 1,-1\cdot 1,-1\cdot -1)=\Sigma(1,-1,-1\cdot -1),
\]
where the second equality uses the fourth bulleted hypothesis, and the third equality is from applying left reorderability to the given families and map.  Similarly, taking $\underline{a}=(1,-1)$ and $\underline{r}=(1,1,-1)$, and defining $\psi\colon K\to \power(I)$ by the rule that $0\mapsto \{0\}$ and $1,2\mapsto \{1\}$, then
\[
1=1+0=1\cdot 1+[1+-1]\cdot -1=\Sigma(1\cdot 1,1\cdot -1,-1\cdot -1)=\Sigma(1,-1,-1\cdot -1).
\]
Putting the previous two displayed equations together finishes the proof that $(1,1)$ is summable.

Next, we show additive extension closure, Axiom \ref{Axiom:AdditiveExtClosure}.  Let $\underline{b}\in X^J$ be summable, let $x\in X$, and let $\ell$ be an index not in $J$.  Take $\underline{r}:=\underline{b}\#_{\ell}x$ (with the hash as defined in the proof of Theorem \ref{Thm:FiniteExtensions}), so that $K=J\cup\{\ell\}$.  Also take $\underline{a}=(1,1)$.  Define $\psi\colon K\to \power(I)$ by the rule $\ell\mapsto \{1\}$ and $j\mapsto \{0\}$ for each $j\in J$.  Applying left reorderability with these choices, then \eqref{Eq:MainLeftReorder} simplifies to the equality
\[
\Sigma\underline{b}+x=\Sigma [\underline{b}\#_{\ell}x].
\]
Thus, the domain of $\Sigma$ is closed under singleton extensions, and hence finite extensions.  In particular, the induced addition is total.  Moreover, by insertive associativity on triples (and some reindexing), the induced addition is associative.  It is commutative by reindexing invariance.  We already noted that $\Sigma()$ is an additive identity and that every element $x\in X$ has an additive inverse $-1x$.  Two more easy applications of left reorderability show that multiplication left and right distributes over the induced addition.  The last sentence stated in this theorem now follows.  Prefix associativity also follows, as do the axioms of Section \ref{Section:TopologyRecovered}.

We have also shown finite-totality.  Note that $I$-totality will not hold whenever $I$ is infinite, by Theorem \ref{Thm:WeakSwindle}, as long as $X$ has more than one element.

We verify addition functoriality by showing the equivalent condition in Proposition \ref{Prop:AddFunct}.  Let $\underline{b}\in X^{J_1}$ and $\underline{c}\in X^{J_2}$ be summable families with disjoint index sets, and take $\underline{r}=\underline{b}+\underline{c}$.  Also take $\underline{a}=(1,1)$.  Define $\psi\colon K\to \power(I)$ by the rule
\[
k\mapsto \begin{cases}
\{0\} & \text{ if $k\in J_1$,}\\
\{1\} & \text{ otherwise.}
\end{cases}
\]
Applying left reorderability, the outer sum on the right side of \eqref{Eq:MainLeftReorder} is over the family $\underline{r}$, showing the needed summability condition.

Next we verify negation functoriality.  Take $\underline{r}=(-1)$.  Also take $\underline{a}\in X^I$ to be an arbitrary summable family, for some index set $I$.  Define $\psi\colon K\to \power(I)$ by the rule $0\mapsto I$.  Applying left reorderability, then \eqref{Eq:MainLeftReorder} holds and simplifies to the needed equality.

For left multiple summability, let $\underline{a}\in X^I$ again be an arbitrary summable family, for some index set $I$.   Take $K=I$, and let $\underline{r}\in X^K=X^I$ be arbitrary.  Define $\psi\colon K\to \power(I)$ by the rule that $k\mapsto \{k\}$.  Now, \eqref{Eq:MainLeftReorder} holds, and the left side is exactly the sum over the family $(r_ia_i)_{i\in I}$.

Finally, we verify infinite distributivity.  Let $\underline{a}\in X^I$ and $\underline{b}\in X^J$ be summable families.  Take $K=I\times J$ and take $r_{(i,j)}=b_j$ for each $(i,j)\in I\times J$.  Define $\psi\colon K\to \power(I)$ by the rule $(i,j)\mapsto \{i\}$.  Applying left reorderability, then \eqref{Eq:MainLeftReorder} holds and simplifies to
\[
\Sigma\left(\left[\Sigma \underline{b}\right]a_i\right)_{i\in I}=\Sigma(b_j a_i)_{(i,j)\in I\times J}.
\]
Another use of left reorderability, this time taking $\underline{r}=(\Sigma\underline{b})$ with $K=\{0\}$, and defining $\psi\colon K\to \power(I)$ by the rule $0\mapsto I$, shows that
\[
\Sigma\left(\left[\Sigma \underline{b}\right]a_i\right)_{i\in I}=\Sigma([\Sigma\underline{b}][\Sigma\underline{a}])_{k\in K}=[\Sigma\underline{b}][\Sigma\underline{a}].
\]
Putting the previous two displayed equations together gives us the needed equality.
\end{proof}

The hypotheses in Theorem \ref{Thm:MainBigLeftReord} are all satisfied by the endomorphism summation system of Example \ref{Example:Endos}.  The left reorderability axiom captures a lot of the inherent algebraic structure in such systems.  Thus, we make the following definition, generalizing Example \ref{Example:Endos}.

\begin{definition}
Any set $X$ (with a summation system $\Sigma$ and monoid structure) satisfying all the hypotheses of Theorem \ref{Thm:MainBigLeftReord} is a (left) \emph{reorderable ring}.
\end{definition}

Right reorderable rings may be defined symmetrically, but we will never have need to make use of them.  Thus, we will drop the left-right adjectives completely, leaving it implicitly understood that we always work on the left.

There is another way to view the reordering axiom.  In the presence of distributivity (and when singletons sum simply), the two sides of \eqref{Eq:MainLeftReorder} simplify to
\[
\Sigma\left(\Sigma(r_ka_i)_{k\in \psi'(i)}\right)_{i\in I} = \Sigma\left(\Sigma (r_ka_i)_{i\in \psi(k)}\right)_{k\in K},
\]
expressing the fact that the order of summation may be interchanged.  These double sums can also be written as a single sum, as follows.

\begin{cor}\label{Cor:ReordCor}
Let $X$ be a reorderable ring.  Let $\underline{a}\in X^I$ be a summable family, let $\underline{r}\in X^K$ be arbitrary, and let $\psi\colon K\to \power(I)$ be a map such that $(r_k)_{k\in \psi'(i)}$ is summable for each $i\in I$.  Taking $L:=\{(j,k)\in I\times K\, :\, j\in \psi(k)\}$, then
\[
\Sigma\left(\Sigma (r_k)_{k\in \psi'(i)}a_i \right)_{i\in I} = \Sigma (r_ka_j)_{(j,k)\in L}.
\]
\end{cor}
\begin{proof}
Define a map $\rho\colon L\to \power(I)$ by the rule $(j,k)\mapsto \{j\}$, and let $\rho'\colon I\to \power(L)$ be the dual map.  Thus, for each $i\in I$ we find
\begin{align*}
\rho'(i) & =  \{(j,k)\in L\, :\, i\in \rho((j,k))\} = \{(j,k)\in L\, :\, i\in \{j\}\} = \{i\}\times \psi'(i).
\end{align*}
Consider the family $\underline{s}\in X^L$ where $s_{(j,k)}=r_k$ for each $(j,k)\in L$.  For each $i\in I$, we have
\[
(s_{(j,k)})_{(j,k)\in \rho'(i)}=(r_k)_{(j,k)\in \{i\}\times\psi'(i)}.
\]
This is a reindexing of the family $(r_k)_{k\in \psi'(i)}$, which is summable by hypothesis.  Thus,
\begin{align*}
\Sigma\left(\Sigma (r_k)_{k\in \psi'(i)}a_i \right)_{i\in I} & =\Sigma\left(\Sigma(s_{(j,k)})_{(j,k)\in \rho'(i)}a_i\right)_{i\in I}\\ &= \Sigma\left(s_{(j,k)}\Sigma (a_i)_{i\in \rho((j,k))}\right)_{(j,k)\in L} = \Sigma (r_ka_j)_{(j,k)\in L},
\end{align*}
where the second equality uses Axiom \ref{Axiom:LeftReorder}, and the third equality uses the fact that singletons sum simply.
\end{proof}

Before ending this section, we mention two more facts about $\Sigma$ when it is the summation system for endomorphism rings.  Both results concern reconstructing information in $\Sigma$ from special subsets.  First, it turns out that we can reconstruct $\dom(\Sigma)$ from $\dom(\Sigma)\cap X^{\N}$.

\begin{prop}
For endomorphism summation, a family is summable if and only if all countable subfamilies are summable.
\end{prop}
\begin{proof}
We already know that subfamilies of summable families are summable, so it suffices to prove the backwards direction.  We will do so contrapositively.  Let $(a_i)_{i\in I}$ be a non-summable family of endomorphisms on a module $M$.  There must then exist some $m\in M$ such that there is an infinite subset $J\subseteq I$ with $a_j(m)\neq 0$ for each $j\in J$.  Restricting $J$ to a countable subset of itself, then $(a_{j})_{j\in J}$ is not summable.
\end{proof}

The previous proposition is about $\dom(\Sigma)$, but does not mention the actual sums in $\im(\Sigma)$.  It seems like a difficult, if not impossible, problem to reconstruct sums just from knowing which families are summable.  Miraculously, we can reconstruct countable sums (in endomorphism rings, and more generally) from finite sums and the set $\dom(\Sigma)\cap X^{\N}$.

\begin{prop}[{cf.\ \cite[Lemma 2.7]{Corner}}]\label{Prop:CountIm}
If $X$ is a reorderable ring, then a family $\underline{a}\in X^{\N}$ is summable with sum $t$ if and only if $(t-(a_0+\cdots + a_{k-1}))_{k\in \N}$ is summable.
\end{prop}
\begin{proof}
$(\Rightarrow)$: Assume $\Sigma \underline{a}=t$.  Let $K=\N$, and for each $k\in K$ let $r_{k}=1$ (the multiplicative identity in $X$).  Define $\psi\colon \N\to \power(\N)$ by the rule $k\mapsto \N_{\geq k}$.  Thus, the dual $\psi'\colon \N\to \power(\N)$ is given by the rule $\psi'(i)=\N_{\leq i}$.  For each $i\in \N$, the family $(r_k)_{k\in \psi'(i)}$ is summable, being finite.  By Axiom \ref{Axiom:LeftReorder},
\begin{equation}\label{Eq:FirstReordering}
\Sigma([\Sigma(r_k)_{k\in \psi'(i)}]a_i)_{i\in \N}=\Sigma(1\Sigma(a_{\ell})_{{\ell}\in \psi(k)})_{k\in K}=\Sigma(\Sigma(a_{\ell})_{\ell\in \N_{\geq k}})_{k\in \N},
\end{equation}
and so all sums on the right side are defined.  Applying prefix associativity (without the shift reindexing of the tail) a total of $k$ times to $\underline{a}$, we find
\begin{equation}\label{Eq:SecondReordering}
t=(a_{0}+\cdots +a_{k-1}) + \Sigma(a_{\ell})_{\ell\in \N_{\geq k}}.
\end{equation}
Rearranging \eqref{Eq:SecondReordering} and plugging in to \eqref{Eq:FirstReordering} shows that the $(t-(a_0+\cdots + a_{k-1}))_{k\in \N}$ is summable.

$(\Leftarrow)$: Assume $\underline{b}:=(t-(a_0+\cdots + a_{k-1}))_{k\in \N}$ is summable.  By prefix associativity (this time \emph{with} the shift reindexing of the tail), we know that $\underline{b}':=(t-(a_0+\cdots + a_{k-1}+a_k))_{k\in \N}$ is also summable.  By addition and negation functoriality, $\underline{b}-\underline{b}'$ is summable.  But $\underline{b}-\underline{b}'=\underline{a}$.

Let $s:=\Sigma\underline{a}$.  By this proof's forward direction, $\underline{c}:=(s-(a_0+\cdots + a_{k-1}))_{k\in \N}$ is summable.  By another use of addition and negation functoriality, $\underline{c}-\underline{b}$ is summable.  This is just the constant family $(s-t)_{k\in \N}$.  By Theorem \ref{Thm:WeakSwindle} we get $s=t$.
\end{proof}

\section{Factor algebras and summation systems}\label{Section:FactorAlg}

In this closing section of the paper, we investigate how summation systems can pass to quotient structures.  Even if a summation system arose from a topology, this does not mean, \emph{a priori}, that the quotient structure arises from a quotient topology.  The algebraic generality of our constructions allows potentially new structures to emerge.

Let $(X,+)$ be an abelian group, and let $S\leq X$ be a (normal) subgroup.  Let $\overline{X}=X/S$ be the factor group.  Throughout this section, we will use bar notation to denote working in this group of cosets.

Working this way induces a natural quotient relation
\[
\overline{\Sigma} = \left\{\left((\overline{a_i})_{i\in I},\overline{x}\right)\in \bigcup_{I}\overline{X}^I\times \overline{X} \, :\, \text{$(a_i)_{i\in I}$ is $\Sigma$-summable, with sum $x$} \right\}.
\]
In order for $\overline{\Sigma}$ to be a summation system on $\overline{X}$, it must be the case that $\overline{\Sigma}$ is a function, not merely a relation.

To understand the contrary situation,  suppose for a moment that $\overline{\Sigma}$ is not a function.  We can then fix pairs
\[
((a_i)_{i\in I},x),((b_i)_{i\in I},y)\in \Sigma
\]
with $(\overline{a_i})_{i\in I}=(\overline{b_i})_{i\in I}$, but with $\overline{x}\neq \overline{y}$.  If $\Sigma$ satisfies addition and negation functoriality, then the two families $\underline{a}-\underline{a}$ and $\underline{a}-\underline{b}$ are both summable.  Moreover, modulo $S$ they are both the constant $\overline{0}$ family, but modulo $S$ their sums are $\overline{x}-\overline{x}=\overline{0}$ and $\overline{x}-\overline{y}\neq \overline{0}$, respectively.

Thus, when $\Sigma$ respects subtraction, a sufficient condition for $\overline{\Sigma}$ to be a function is that $S$ satisfies the following condition.

\begin{definition}
Let $X$ be a set with a summation system $\Sigma$.  A subset $S\subseteq X$ is \emph{$\Sigma$-closed} if whenever all members of a summable family belong to $S$, then the sum belongs to $S$.
\end{definition}

Unsurprisingly, this condition is also necessary.

\begin{prop}\label{Prop:SigmaClosedQuot}
Let $(X,+)$ be an abelian group with a summation system $\Sigma$ satisfying addition and negation functoriality. Given a subgroup $S\leq X$, then the quotient relation $\overline{\Sigma}$ is a summation system on $\overline{X}$ if and only if $S$ is $\Sigma$-closed.
\end{prop}
\begin{proof}
We already proved the backwards direction, contrapositively.  We will now show the forward direction.  Assume that $\overline{\Sigma}$ is a function.  Let $\underline{a}=(a_i)_{i\in I}\in S^{I}$ be a summable family, with sum $x\in X$.  Our goal is to show that $x\in S$.

Notice that $(\overline{a_i})_{i\in I}=(\overline{0})_{i\in I}$ has $\overline{\Sigma}$-sum $\overline{x}$.  On the other hand, since subtraction is respected, $\underline{a}-\underline{a}=(0)_{i\in I}$ has $\Sigma$-sum $x-x=0$.  Thus, $\overline{x}=\overline{\Sigma}(\overline{0})_{i\in I}=\overline{0}$, and so $x\in S$.
\end{proof}

Whenever summable families are ordinal-indexed, more can be said about the concept of $\Sigma$-closure, connecting it to a topology, or at least a topological-like conditions.

\begin{lemma}\label{Lemma:BicondSigmaClosed}
Let $(X,+)$ be an abelian group with a subgroup $S$.  Let $\Sigma$ be a summation system on $X$ whose summable families are ordinal-indexed, and with $\Sigma$ satisfying \textup{Axiom \ref{Axiom:InitialSegments}}.  If $S$ is $\Sigma$-closed, then for each $\underline{a}\in \dom(\Sigma)$,
\[
\underline{a}\in S^I\, \Longleftrightarrow\,  p(\underline{a})\in S^I.
\]
The same biconditional is true if $S$ is closed under $\Sigma$-limits.  Consequently, a subgroup of $X$ is $\Sigma$-closed if and only if it is closed under $\Sigma$-limits.
\end{lemma}
\begin{proof}
Let $S$ be a $\Sigma$-closed subgroup of $X$, and let $\underline{a}\in X^I$ be a summable family.

$(\Rightarrow)$: Assume that all members of $\underline{a}$ belong to $S$.  By Axiom \ref{Axiom:InitialSegments}, all partial sums are defined.  Further, each partial sum $s_i=\Sigma (a_j)_{j\in I_{<i}}+a_i$ belongs to $S$, using the fact that $S$ is closed under sums and under addition.

$(\Leftarrow)$:  Assume that all members of $p(\underline{a})$ belong to $S$.  By transfinite induction, assume that $a_j\in S$ for all indexes $j<i\in I$.  The partial sum $\Sigma (a_j)_{j\in I_{<i}} + a_i$ is in $S$, by assumption.  Also, $\Sigma (a_j)_{j\in I_{<i}}\in S$ by the inductive assumption in conjunction with the $\Sigma$-closed condition on $S$.  Finally, since $S$ is a subgroup, it is closed under subtraction, and so $a_i\in S$.

A similar, dualized argument shows that the biconditional holds when $S$ is closed under $\Sigma$-limits.  The final sentence is also clear, since the $\Sigma$-limit of $p(\underline{a})$ is just the $\Sigma$-sum of $\underline{a}$, and there is a natural bijection between the two types of families by Proposition \ref{Prop:LimitsPartialSums}.
\end{proof}

\begin{cor}\label{Cor:ClosureInduced}
Let $(X,+)$ be an abelian group.  Let $\tau$ be a Hausdorff topology on $X$, and let $\Sigma$ be the induced summation system.  A subgroup $S\leq X$ is $\Sigma$-closed if and only if it is closed in the $\tau_{\Sigma}$ topology.
\end{cor}
\begin{proof}
Being an induced summation system, the summable families under $\Sigma$ are ordinal-indexed.  Also, Axiom \ref{Axiom:InitialSegments} holds by Theorem \ref{Thm:WhatPartialSummationGetsUs}(3).  Thus, Lemma \ref{Lemma:BicondSigmaClosed} applies.  So, it suffices to show that $S$ is closed under $\Sigma$-limits if and only if it is $\tau_{\Sigma}$-closed.

Note that according to Proposition \ref{Prop:FinestTopologyWithLimitsDetermined} and Definition \ref{Def:SigmaTopology} a set $S\subseteq X$ is $\tau_{\Sigma}$-closed exactly when, for any sequence $\underline{s}$ with a $\Sigma$-limit $x$, if a cofinal subsequence of $\underline{s}$ belongs to $S$, then so does $x$.  In particular, any such set is closed under $\Sigma$-limits, since a sequence is always cofinal in itself.

Conversely, suppose that $S$ is closed under $\Sigma$-limits.  Let $\underline{s}$ be a sequence with a $\Sigma$-limit $x$, and suppose that a cofinal subsequence $\underline{t}$ belongs to $S$.  Since $\tau$ is Hausdorff, $\Sigma$ satisfies \eqref{Eq:SubsequenceProperty} by Theorem \ref{Thm:WhatPartialSummationGetsUs}(8).  Therefore, $\underline{t}$ has $x$ as a $\Sigma$-limit, and so $x\in S$, as desired.
\end{proof}

In the previous corollary, the assumption that $\Sigma$ is the induced summation system for a Hausdorff topology could be weakened a bit.  We need enough reindexing so that every summable family under $\Sigma$ can be replaced by an ordinal-indexed family, in which case if Axiom \ref{Axiom:InitialSegments} and \eqref{Eq:SubsequenceProperty} both hold, then the conclusion of the corollary holds as well.

The previous work applies just as well to factor rings as it does to factor groups.  In particular, all of the previous results of this section apply to ideals in reorderable rings.  In the following example, we illustrate how $\Sigma$-closure behaves in endomorphism rings of vector spaces, relative to various reorderable summation systems.

\begin{example}
Let $V_D$ be a right vector space of infinite dimension $\kappa\geq \aleph_0$, over a division ring $D$.  Let $X$ be its endomorphism ring.  It is well-known that we can view $X$ as the ring ${\rm CFM}_{\kappa}(D)$ of $\kappa\times \kappa$ column-finite matrices with entries in $D$.  (A matrix is \emph{column-finite} when each column has only finitely many nonzero entries.)

Let $\Sigma$ be the summation system induced by the finite topology.  A family
\[
\underline{a}=(a_i)_{i\in I}\in X^I
\]
of column-finite matrices is summable exactly when for each fixed column (index) only finitely many of the $a_i$ have nonzero entries in that column.  The sum $\Sigma\underline{a}$ is then easy to compute; any given column is determined by adding together the corresponding columns of each $a_i$.

The ideals of $X$ are exactly the zero ideal, the trivial ideal $X$, and for each infinite cardinal $\lambda\leq \kappa$ the set $J_{<\lambda}$ of column-finite matrices with row rank less than $\lambda$.  (So, when $\kappa=\aleph_0$, there is only one nontrivial nonzero ideal, consisting of matrices with finite rank.)  This characterization of all the ideals of ${\rm CFM}_{\kappa}(D)$ is well-known, and for a fully worked-out proof see the solution to \cite[Ex.\ 3.16]{LamExercises}.

The only $\Sigma$-closed ideals are the zero ideal and the trivial ideal.  Indeed, if we let $e_{i,j}$ denote the matrix with $1$ in the $(i,j)$ entry and zeros elsewhere, then just consider the family $(e_{i,i})_{i\in \kappa}$, which is summable with sum the identity matrix $1\in X$.  On the other hand, there are many $\Sigma$-closed left ideals; just take the annihilator of a subset of $V$.

We can modify $\Sigma$ to create a new summation system that is still reorderable.  Given an infinite cardinal $\mu\leq \kappa$, let $\Sigma_{<\mu}$ be the restriction of $\Sigma$ to those summable families with less than $\mu$ nonzero entries.  For instance, $\Sigma_{<\aleph_0}$ is just the restriction of $\Sigma$ to finite families (up to core-extensions).  The $\Sigma_{<\mu}$-closed ideals of $R$ are exactly the zero ideal, the trivial ideal, and the ideals $J_{<\lambda}$ with $\lambda\geq \mu$.  Consequently, we can make sense of infinite summation in $R/J_{<\lambda}$ as long as we are only summing fewer than $\lambda$ nonzero elements; otherwise the summation is not well-defined.
\end{example}

Not only can summation systems pass to factor algebras, sometimes axioms continue to hold for those quotient systems.  For instance, reorderability automatically passes to factor rings modulo $\Sigma$-closed ideals.  Here follows the quick proof.

\begin{thm}\label{Thm:ReorderablePassQuotients}
If $X$ is a reorderable ring with summation system $\Sigma$, then $\overline{X}$ is a reorderable ring with summation system $\overline{\Sigma}$ when working modulo a $\Sigma$-closed ideal.
\end{thm}
\begin{proof}
Of the four conditions defining reorderable rings, three immediately pass to factor rings.  It thus suffices to verify that $\overline{\Sigma}$ satisfies Axiom \ref{Axiom:LeftReorder}.

Let $(\overline{a_i})_{i\in I}$ be $\overline{\Sigma}$-summable, let $(\overline{r_k})_{k\in K}$ be an arbitrary family of elements from $\overline{X}$, and let $\psi\colon K\to \power(I)$ be a map such that $(\overline{r_k})_{k\in \psi'(i)}$ is summable for each $i\in I$.  Summable families from $\overline{X}$ are exactly the images of summable families from $X$.  So, we may assume that $(a_i)_{i\in I}$ is a $\Sigma$-summable family lifting $(\overline{a_i})_{i\in I}$.  Also, for each $i\in I$ we can lift the summable family $(\overline{r_k})_{k\in \psi'(i)}$ to a summable family $(r_{(i,k)})_{k\in \psi'(i)}$.  Unfortunately, we may not be able to piece these lifted families together into a single lift of the family $(\overline{r_k})_{k\in K}$, which is what makes this part of the proof complicated.

We now have
\[
\Sigma\left(\Sigma \left(r_{(i,k)}\right)_{k\in \psi'(i)} a_i\right)_{i\in I} = \Sigma \left(r_{(i,k)}a_i\right)_{\{(i,k)\in I\times K\, :\, k\in \psi'(i)\}}
\]
by an application of Corollary \ref{Cor:ReordCor}.  Using insertive associativity to partition the sum on the right side, it is equal to
\[
\Sigma \left(\Sigma \left(r_{(i,k)}a_i\right)_{i\in \psi(k)}\right)_{k\in K}.
\]
Thus, passing to the factor ring, and noting that $\overline{r_{(i,k)}}=\overline{r_k}$ for each $k\in K$ and each $i\in \psi(k)$, we obtain
\[
\overline{\Sigma}\left(\overline{\Sigma} \left(\overline{r_{k}}\right)_{k\in \psi'(i)} \overline{a_i}\right)_{i\in I} = \overline{\Sigma} \left(\overline{\Sigma} \left(\overline{r_{k}a_i}\right)_{i\in \psi(k)}\right)_{k\in K}.
\]
To finish, it suffices to show that $\overline{\Sigma} \left(\overline{r_{k}a_i}\right)_{i\in \psi(k)}=\overline{r_k}\,\overline{\Sigma} \left(\overline{a_i}\right)_{i\in \psi(k)}$, for each $k\in K$.  This follows easily from the fact that the same equality without bars holds in the reorderable ring $X$.
\end{proof}

It would be interesting to know if there are other algebraic axioms that hold for endomorphism ring summation that are not implied by reorderability, especially any which also pass to factor algebras.

We finish this paper by showing how the Jacobson radical is often closed under countable sums.  (The reader is directed to \cite{Lam} for basic facts about the Jacobson radical.)

\begin{thm}[{cf.\ \cite[Theorem 8.3]{Corner}}]\label{Thm:ReorderableCountableJacobson}
Let $X$ be a reorderable ring with respect to a summation system $\Sigma$.  If $X$ is an $I_0$-ring, meaning that every left ideal of $X$ without nonzero idempotents is contained in the Jacobson radical $J$, then $J$ is closed under countable $\Sigma$-sums.
\end{thm}
\begin{proof}
Let $\underline{a}=(a_i)_{i\in \N}\in J^{\N}$ be a summable family, with sum $x\in X$.  Let $e\in Rx$ be an idempotent.  It suffices to show that $e=0$.  After replacing $\underline{a}$ by a left multiple, we may as well assume that $\Sigma \underline{a}=e$.  Moreover, after multiplying on the left and the right by $e$, we may assume that the members of $\underline{a}$ live in $eJe$, which is the Jacobson radical of $eXe$.

Let $s_i:=\Sigma (a_j)_{j\in \N_{\leq i}}\in eJe$.  By Proposition \ref{Prop:CountIm}, we have $(e-s_{i-1})_{i\in \N}$ is summable.  Each member is a unit in the corner ring $eXe$, so after replacing $(e-s_{i-1})_{i\in \N}$ by the appropriate left multiple family, each member is now equal to $e$.  By Theorem \ref{Thm:WeakSwindle}, we must have $e=0$.
\end{proof}

\begin{question}\label{Question:Final}
Does the previous theorem hold for uncountable sums?
\end{question}

There are some cases where the answer is positive.  To start, by \cite[Theorem 8.3]{Corner}, if $X$ is an $I_0$-ring that is a topological ring with respect to a Hausdorff, left linear topology $\tau$, then the Jacobson radical $J$ is $\tau$-closed.  When $X$ is additionally an endomorphism ring, when $\tau$ is the finite topology, and when $\Sigma$ is the induced topology, then $\tau\subseteq \tau_{\Sigma}$, so $J$ is $\tau_{\Sigma}$-closed, and hence $J$ is $\Sigma$-closed by Corollary \ref{Cor:ClosureInduced}.  This same argument also applies to $\overline{X}=X/S$ with quotient system $\overline{\Sigma}$, when $S$ is a closed ideal in the finite topology.

There is still the possibility that when $S$ is merely a $\Sigma$-closed ideal, the Jacobson radical of $\overline{X}$ may not be $\overline{\Sigma}$-closed.  This is because $S$ is $\tau_{\Sigma}$-closed, but may fail to be $\tau$-closed (since $\tau$ is generally coarser than $\tau_{\Sigma}$).  Hence, $\overline{X}$ may no longer be a Hausdorff topological ring.  If this happens when $X$ is an \emph{exchange ring} in the sense of \cite{WarfieldExchange} (which is a stronger condition than being $I_0$), this would answer an old open question raised in \cite{CJ} (whether ``finite exchange'' implies ``infinite exchange'') in the negative.

It may be easier to answer the following question:

\begin{question}
If $X=\End(M_R)$, and if $\Sigma$ is endomorphism ring summation on $X$, is every $\Sigma$-closed ideal also closed in the finite topology?
\end{question}

\section*{Acknowledgements}

The author thanks Michael Hru\v{s}\'{a}k and Alexander Shibakov for answering questions, which ultimately helped in the formulation of the example at the end of Section \ref{Section:UncondSums}.  The author also thanks George Bergman, Greg Conner, Curt Kent, Kyle Pratt, and Manny Reyes for comments and suggestions that improved the paper.  This work was partially supported by a grant from the Simons Foundation (\#963435 to Pace P.\ Nielsen).

\providecommand{\MR}{\relax\ifhmode\unskip\space\fi MR }
\providecommand{\MRhref}[2]{%
  \href{http://www.ams.org/mathscinet-getitem?mr=#1}{#2}
}
\providecommand{\href}[2]{#2}

\end{document}